%% file: article_gcd.tex
\documentclass[a4paper,12pt,reqno]{amsart}
\input{preamble}
\include{contents}
\usepackage{xcolor}
\definecolor{red}{rgb}{0,0,0} 

\begin{document}
\author{Edison H L Au-Yeung}
\title[An upper bound for valuations of elliptic nets at primes of bad reduction]{An upper bound for valuations of elliptic nets at primes of bad reduction}
\address{Mathematics Institute, Zeeman Building, University of Warwick, Coventry CV4 7AL, England}
\email{hang.au-yeung@warwick.ac.uk}
\date{\today}
\thanks{The author is supported by the Warwick Mathematics Institute Centre for Doctoral Training, and gratefully acknowledges funding from the University of Warwick.}

\begin{abstract}
Let $E/K$ be an elliptic curve over a field defined by the short Weierstrass equation and let $\bm{P}=(P_1, P_2) \in E(K)^2\setminus \{(O, O), (\pm P, \pm P)\}$ be a pair of non-torsion points, where $O$ denotes the point at infinity. We give a uniform upper bound for the quantity $\Min(\nu(\Phi_{(v_1, v_2)}(\bm{P})), 2\nu(\Psi_{(v_1, v_2)}(\bm{P})))$, where $\nu$ is a normalised discrete valuation associated to primes where $E$ has bad reduction and $\Psi^2_{(v_1, v_2)}(\bm{P})$, $\Phi_{(v_1, v_2)}(\bm{P})$ are rational functions for describing $\Z$--linear combination of points $[v_1]P_1+[v_2]P_2$ on $E$, higher--dimensional analogues of division polynomials $\psi^2_n, \phi_n$ known as net polynomials. They are also elliptic nets, which are $n$--dimensional arrays with values in $K$ satisfying the same nonlinear recurrence relation that division polynomials do as well. Hence, they are useful in studying counting problems on elliptic curves. 
\end{abstract}
\maketitle
\setcounter{tocdepth}{1}
\tableofcontents
\input{Introduction}
\input{Review}
\input{Elliptic_nets_and_choice_of_points}

\input{Resultant_and_net_polynomials}

\input{Main_result}

\subsection*{Acknowledgements}
I would like to express my gratitude to my supervisors, Dr.\ Simon L.\ Rydin Myerson, and Dr.\ Helena Verill for their incredible support over the past year, along with their insightful comments, and for carefully reading earlier versions of this paper.
\bibliographystyle{abbrv}
\bibliography{Reference2}

\end{document}

%% file: preamble.tex
\usepackage{geometry}
\usepackage[utf8]{inputenc}

\usepackage{amsfonts}  
\usepackage{amsthm}    
\usepackage{amsmath}   
\usepackage{bm}        
\usepackage{bbm}
\usepackage{amssymb}   
\usepackage{commath}	

\usepackage[dvipsnames, table]{xcolor}	
\usepackage{graphicx}  
\usepackage{tikz}      
\usepackage{tikz-cd}   
\usepackage{pgfplots}  
\pgfplotsset{compat=1.18}
\usetikzlibrary{arrows.meta, positioning, calc, math}  
\tikzset{>={Classical TikZ Rightarrow[width=7pt, length=5pt]}}
\usetikzlibrary{cd}  

\usepackage{array}     
\usepackage{xtab}      
\usepackage{url}

\usepackage{comment}
\usepackage[T1]{fontenc} 
\usepackage{lmodern} 
\usepackage{multicol}
\usepackage{pdflscape, adjustbox}	
\usepackage{parskip}       
\usepackage{mathrsfs}      
\usepackage{fancyhdr}      
\usepackage[hidelinks]{hyperref}  
\usepackage{mathtools}     
\usepackage[shortlabels]{enumitem}  
\usepackage{witharrows}    
\usepackage{float}         
\usepackage{makecell}
\usepackage{cleveref}	

\newcommand{\Z}{\ensuremath{\mathbb{Z}}}
\newcommand{\Q}{\ensuremath{\mathbb{Q}}}
\newcommand{\N}{\ensuremath{\mathbb{N}}}
\newcommand{\R}{\ensuremath{\mathbb{R}}}
\newcommand{\C}{\ensuremath{\mathbb{C}}}
\newcommand{\F}{\ensuremath{\mathbb{F}}}
\newcommand{\q}{\ensuremath{\mathfrak{q}}}
\newcommand{\p}{\ensuremath{\mathfrak{p}}}
\newcommand{\OO}{\ensuremath{\mathcal{O}}}
\newcommand{\V}{\ensuremath{\mathbb{V}}}

\DeclareMathOperator{\Min}{min}

\DeclareMathOperator{\Gcd}{gcd}

\DeclareMathOperator{\Deg}{deg}

\DeclareMathOperator{\Res}{Res}
\DeclareMathOperator{\Det}{det}
\DeclareMathOperator{\Syl}{Syl}
\DeclareMathOperator{\Weight}{weight}

\newtheorem{prop}{Proposition}[section]

\newtheorem{theorem}[prop]{Theorem}
\newtheorem{lemma}[prop]{Lemma}

\theoremstyle{definition}
\newtheorem{remark}[prop]{Remark}
\newtheorem{example}[prop]{Example}

\newtheorem{defn}[prop]{Definition}

\numberwithin{equation}{section}

\crefname{prop}{Proposition}{Propositions}
\Crefname{prop}{Proposition}{Propositions}

\crefname{cor}{Corollary}{Corollaries}
\Crefname{cor}{Corollary}{Corollaries}

\crefname{defn}{Definition}{Definitions}
\Crefname{defn}{Definition}{Definitions}

%% file: Introduction.tex
\section{Introduction}
\subsection{Motivation}
Let $E/K$ be an elliptic curve defined over a $K$ with the following integral short Weierstrass equation:
\begin{equation}\label{short EC}
E\colon y^2 = x^3 + Ax + B;\, A, B \in \OO_K.
\end{equation}

Let $P=(x, y)$ be a non-torsion point on $E(K)$, $n$ an integer. Then we can express the point $nP$ with polynomials $\phi_n, \psi_n\in K[x, y]$:
\begin{equation}\label{expression}
x(nP)=\frac{\phi_{n}(P)}{\psi_{n}^{2}(P)}.
\end{equation}
The polynomial $\psi_n$ is called the $n$--th division polynomial associated to $E$. 

Now let $K$ be a field with a discrete valuation $\nu$, let $\OO_{\nu}=\{x\in K\colon \nu(x)\geq 0\}$ and $\p=\{x\in K\colon \nu(x)>0\}$. In \cite[Theorem A]{AyadS}, Ayad proved the following result regarding the valuation of $\psi_{n}$. 

\begin{theorem}[Ayad]\label{Ayad TFAE}
Let $\nu$ and $\p$ be as above. Let $E/K$ be an elliptic curve defined by the equation in (\ref{short EC}). Let $P \in E(K)$ be a non--torsion point in $E(K)$ such that $P \not\equiv O \bmod \p$. Then the following assertions are equivalent:
\begin{enumerate}[(a)]
    \item $\nu(\psi_2(P))$ and $\nu(\psi_3(P)) > 0$.
    \item For all integers $n \geq 2$, we have $\nu(\psi_n(P)) > 0$.
    \item There exists an integer $n_0 \ge 2$ such that $\nu(\psi_{n_0}(P))$ and $\nu(\psi_{n_0+1}(P)) > 0$.
    \item There exists an integer $m_0 \ge 2$ such that $\nu(\psi_{m_0}(P))$ and $\nu(\phi_{m_0}(P)) > 0$.
    \item The reduction of $P$ modulo $\mathfrak{p}$ is singular.
\end{enumerate}
\end{theorem}

In particular, a natural question to ask would be quantifying the common valuation, also known as the \emph{cancellation exponent} of $\psi_{m_0}$ and $\phi_{m_0}$, $\Min\left( \nu\left(\phi_{m_0}(P)\right), 2\nu\left(\psi_{m_0}(P)\right)\right)$ in (d) explicitly for such valuations. In \cite[eq (1.3)]{AyadS}, Ayad gave an answer by mentioning that
\begin{equation}\label{Ayad resultant}
\Res(\phi_n, \psi_n^2)=\Delta^{\frac{1}{6}n^2(n^2-1)},
\end{equation}
where $\Delta=-16(4A^3+27B^2)$ denotes the discriminant of the curve $E$ defined by (\ref{short EC}). For the proof of (\ref{Ayad resultant}), one can refer to \cite[Lemma 2]{AyadA}, \cite[Lemma 1.7.11(b)]{Connell} or \cite{hindry_cassels} (an analytic proof can also be found in \cite{Schmidt}). Since $P$ is singular modulo $\p$, we know that $\nu(x(P))\geq 0$. By properties of the resultant, we can then deduce that for such valuations,
\begin{equation}\label{rank 1 discriminant bound} 
\Min\left( \nu\left(\phi_{n}(P)\right), 2\nu\left(\psi_{n}(P)\right)\right)\leq \frac{1}{6}n^2(n^2-1)\nu(\Delta).
\end{equation}

For problems related to integral, or $S$--integral points on elliptic curves, such an upper bound becomes quite useful: suppose $P \in E(\Q)$ is an integral point, then (\ref{Ayad resultant}) tells us that in fact 
\[
\Gcd\left(\phi_n(P), \psi_n^2(P)\right)\leq \Delta^{\frac{1}{6}n^2(n^2-1)}.
\]
Combining with existng explicit formulae for the cancellation exponent of $\phi_n$ and $\psi^2_n$, such as the formulae in \cite{Cheon}, \cite{Stange2}, \cite{Verzobio2} and \cite{Yabuta}, this allows us to produce an upper bound of the naive height of integral points, $h(P)=\log|x(P)|$ with $P \in \Z$, and similarly for $S$--integral points. For more details, readers may refer to \cite{Ingram} for integral points and \cite{AyadS} for $S$--integral points. 

Currently, there are existing generalisations of the above results in literature. In \cite{Stange}, Stange generalised the concept of division polynomials to net polynomials, which can be used to describe $\Z$--linear combination of points. 
\begin{defn}[Net polynomial and elliptic denominator net]\label{net poly def}
For an arbitrary field $K$, consider the affine coordinate ring 
\[
R_{n}=K[x_i, y_i]_{1\leq i \leq n}[(x_{i}-x_{j})^{-1}]/\langle f(x_{i}, y_{i})\rangle_{1\leq i \leq n},
\]
where $f$ is the defining polynomial (\ref{short EC}) for the elliptic curve $E/K$. Let $\bm{P}=(P_{1}, ..., P_{n}) \in E(K)^{n}$ and $\bm{v}=(v_{1}, \dotsc, v_{n}) \in \Z^{n}$. If $P_{i}\neq O$ for all $i$ and $P_i \neq \pm P_j$ for any $i \neq j$, then there exist rational functions $\Psi_{\bm{v}}(\bm{P})$, $\Phi_{\bm{v}}(\bm{P})$, $\bar\Omega_{\bm{v}}(\bm{P}) \in R_{n}$ such that 
\begin{equation}\label{notation}
\bm{v} \cdot \bm{P} = [v_{1}]P_{1} +...+[v_{n}]P_{n}=\left(\frac{\Phi_{\bm{v}}(\bm{P})}{\Psi^{2}_{\bm{v}}(\bm{P})}, \frac{\bar\Omega_{\bm{v}}(\bm{P})}{\Psi^{3}_{\bm{v}}(\bm{P})}\right).
\end{equation}
The polynomial $\Psi_{\bm{v}}$ is defined to be the \emph{$\bm{v}$-th net polynomial}, which is an elliptic net (for proof, see \cite[Theorem 4.1]{Stange}).
\end{defn}
For the derivation of these rational functions, readers may refer to \cite[Section 4]{Stange} for details. On the other hand, Akbary et al.\ in \cite{Akbary} generalised Ayad's results for net polynomials. 
\begin{theorem}\cite[Theorem 1.5]{Akbary}\label{Akbary TFAE}
Let $E/K$ be an elliptic curve defined by the equation in (\ref{short EC}). Let $\bm{P}=(P_1, \ldots, P_n) \in E(K)^n$ be such that $P_i \not\equiv O \bmod \p$ and $P_i\pm P_j \not\equiv O \bmod \p$ for $1\leq i < j \leq n$. Then the following are equivalent:
\begin{enumerate}
    \item[(a)] \text{There exists $1 \le i \le n$, such that} $\nu(\Psi_{2\bm{e_i}}(\bm{P})) > 0 \text{ \textit{and} } \nu(\Psi_{3\bm{e_i}}(\bm{P})) > 0$.
    \item[(b)] \text{There exists $1 \le i \le n$ such that for all $n \ge 2$ we have} $\nu(\Psi_{n\bm{e_i}}(\bm{P})) > 0$.
    \item[(c)] \text{There exists $\bm{v} \in \mathbb{Z}^n$ and $1 \le i \le r$ such that} $\nu(\Psi_{\bm{v}}(\bm{P})) > 0 \text{ and } \nu(\Psi_{\bm{v}+\bm{e_i}}(\mathbf{P})) > 0$.
 \item[(d)] \text{There exists $\bm{v} \in \mathbb{Z}^n$ such that} $\nu(\Psi_{\bm{v}}(\bm{P})) > 0 \text{ \text{and} } \nu(\Phi_{\bm{v}}(\bm{P})) > 0$.
    \item[(e)] \text{There exists $1 \le i \le r$ such that $P_i \pmod{\mathfrak{p}}$ is singular.}
\end{enumerate}
\end{theorem}
In this paper, our goal is to quantify and give an upper bound to the size of the cancellation exponent of net polynomials in the case $n=2$, a result analogous to (\ref{rank 1 discriminant bound}). 
\begin{theorem}\label{main result}
Let $\bm{P}=(P_1, P_2)$ be a pair of points in $E(K)$ with $P_1\neq \pm P_2$ such that $P_1, P_2 \not\equiv O \bmod \p$ and $P_1\pm P_2\not\equiv O \bmod \p$. For all $\bm{v}=(v_1, v_2)\in \Z^2$, let $\nu$ be a normalised discrete valuation associated to a prime $\p$ that divides the discriminant of the curve. Then 
\[
0\leq \Min\left(\nu\left(\Phi_{(v_1, v_2)}(\bm{P})\right), 2\nu\left(\Psi_{(v_1, v_2)}(\bm{P})\right)\right) \leq \left(18(v_1^2+v_2^2)^2-4(v_1^2+v_2^2)\right)\nu(\Delta).
\]
\end{theorem}

\begin{remark}
The assumptions $P_i \not\equiv O \bmod \p$ and $P_i\pm P_j \not\equiv O \bmod \p$ for $1\leq i < j \leq n$ in Akbary's result are used to show that that for all $(v_1, v_2)\in \Z^2$, we have $\nu\left(\Phi_{(v_1, v_2)}(\bm{P})\right), \nu\left(\Phi_{(v_1, v_2)}(\bm{P})\right)\geq 0$ (this is \cite[Proposition 1.6]{Akbary}), therefore the common valuation/cancellation exponent considered in \Cref{Akbary TFAE} is well-defined. For this reason, it is necessary for us to retain a similar technical assumption here. In particular, $P_1\equiv \pm P_2 \bmod \p$ is a case we cannot handle with our method. See \Cref{X=0 impossible} for more details.  
\end{remark}

As in the case of division polynomials, it is expected that such upper bound can be used to tackle problems related to $\Z$--linear combination of integral points $[n]P+[m]Q$. For instance, the author has also proved an explicit formula for the cancellation exponent of net polynomials associated to elliptic curves with complex multiplication in \cite{AY}, which is a prototype for $\Z$--linear combination of integral points. Therefore, producing analogous integral point results using net polynomials, specialised to elliptic curves with complex multiplication, should be within reach.

The paper is organised as follows. After giving some notations, we will give a review of basic properties and definitions of net polynomials, elliptic nets and resultants of polynomials in Section 2. In Section 3, we will give a criterion for when the net polynomials can vanish simultaneously. After discussing some technicalities of our method in Section 4, we will finally prove our main result in Section 5. 

\subsection{Notation}\label{Notations}
Throughout the paper, unless otherwise specified, we will use the following notations.
\begin{table}[h!]
\begin{tabular}{cc}
$K$           & Arbitrary fields       \\
$\OO_{K}$     & The ring of integers of a number field $K$   \\
$\Delta$ 	&The discriminant of a cubic Weierstrass curve $E$ \\
$\nu$	&               A normalised discrete valuation\\ 
\end{tabular}
\end{table}

%% file: Review.tex
\section{Preliminaries}
In this section, we will collect some basic facts about net polynomials and resultants, so the paper is self--contained. 

\subsection{Elliptic net and net polynomials}
An \emph{elliptic divisibility sequence} (EDS) is a sequence $(W_n)$ that takes values in an integral domain and satisfies the recursion 
\[
W_{n+m}W_{n-m}W_{r}^2=W_{n+r}W_{n-r}W_{m}^2-W_{m+r}W_{m-r}W_{n}^2
\]
for any integers $n, m, r$. In \cite{Stange}, Stange generalised the concept of an elliptic sequence to an $n$--dimension array, also known as an \emph{elliptic net}, which extends the recurrence relation above to `a higher rank'. 

\begin{defn}[Elliptic net]\label{elliptic net}
Let $A$ be a free finitely-generated abelian group and $R$ be an integral domain. An \emph{elliptic net} is any map $W \colon A \rightarrow R$ with $W(\bm{0})=0$ and, for any $\bm{p}, \bm{q}, \bm{r}, \bm{s} \in A$,
\begin{multline*}
W(\bm{p}+\bm{q}+\bm{s})W(\bm{p}-\bm{q})W(\bm{r}+\bm{s})W(\bm{r}) \\
+ W(\bm{q}+\bm{r}+\bm{s})W(\bm{q}-\bm{r})W(\bm{p}+\bm{s})W(\bm{p}) \\ 
+ W(\bm{r}+\bm{p}+\bm{s})W(\bm{r}-\bm{p})W(\bm{q}+\bm{s})W(\bm{q})=0.
\end{multline*} \label{elliptic net def}
We call the rank of \(W\) the rank of the elliptic net. 
\end{defn}

This is indeed a generalisation of EDS: note that for an elliptic net $W: A \rightarrow R$, if we take $A=\Z$, $\bm{p}=m, \bm{q}=n, \bm{r}=r, \bm{s}=0$, then $W$ is an EDS by definition since $W$ is an \emph{odd} function.

\begin{lemma}\cite[Proposition 2.1]{Stange} \label{odd net}
Let $A$ be a free abelian group of finite rank, and $R$ be an integral domain. Let $W\colon A \rightarrow R$ be an elliptic net. Then for any $\bm{v} \in A, W(-\bm{v})=-W(\bm{v})$. In other words, an elliptic net is an odd function. 
\end{lemma}

Similar to division polynomials, we can relate elliptic nets to elliptic curves. In \cite{Stange}, Stange has shown that there is a collection of rational functions on elliptic curves over arbitrary fields that satisfy the definition of elliptic nets and are used to describe the collection of denominators of linear combinations of points, $\sum_{i=1}^{n}[m_{i}]P_{i}$, on an elliptic curve, which means we are taking $A=\Z^{n}$ in the definition of an elliptic net. As we have defined in \Cref{net poly def}, these rational functions are known as net polynomials. 

The following properties of net polynomials will be useful to take note of in this paper. We refer the readers to \cite[Lemma 2.5, Lemma 2.6, Theorem 2.8, Proposition 3.8]{Stange} and \cite[Lemma 2.5]{Akbary} for the proof. 
\begin{prop}[Properties of net polynomials]\label{net poly}
The net polynomials satisfy the following properties:
\begin{enumerate}[(a)]
\item $\Psi_{\bm{v}}$ is an elliptic net, so the recurrence in \Cref{elliptic net} holds;
\item For all $\bm{v} \in \mathbb{Z}^{n}$, the values of the $\bm{v}$-th net polynomial are completely defined by the following initial conditions:
\begin{itemize}
\item $\Psi_{\bm{0}} = 0$;
\item $\Psi_{\bm{e}_i} = 1$ for $1\leq i \leq r$;
\item $\Psi_{2\bm{e}_i} = 2y_i + a_{1}x_i + a_{3} = \psi_{2}(P_i)$;
\item $\Psi_{\bm{e}_i + \bm{e}_j} = 1$, $i \neq j$;
\item $\Psi_{2\bm{e}_i + \bm{e}_j} = 2x_i + x_j - \left(\frac{y_j - y_i}{x_j - x_i}\right)^{2}
      - a_{1}\left(\frac{y_j - y_i}{x_j - x_i}\right) + a_{2}$, $i \neq j$.
\end{itemize}
\label{a values}
\item For $1\leq i \leq r$, we have $\Psi_{n\bm{e}_{i}}(\bm{P})=\psi_{n}(P_i)$.
\item For any $\bm{v}, \bm{u} \in \mathbb{Z}^{r}$,
\begin{equation}
\Psi_{\bm{v}}^{2}\Psi_{\bm{u}}^{2}\left(x\left(\bm{v} \cdot \bm{P}\right)
- x\left(\bm{u} \cdot \bm{P}\right)\right)
= -\Psi_{\bm{v} + \bm{u}}\Psi_{\bm{v} - \bm{u}},
\end{equation}
\label{notation}
where $x(\bm{v} \cdot \bm{P})$ represents the $x$-coordinate of the point $\bm{v} \cdot \bm{P}$.

\item Using the notations in \ref{notation}, for $1 \leq i \leq r$, substitute $\bm{u} = \bm{e}_i$ in \ref{notation}; then we have
\begin{equation}
\Phi_{\bm{v}}(\bm{P})
= \Psi_{\bm{v}}^{2}(\bm{P})\,x(P_i)
- \Psi_{\bm{v} + \bm{e}_i}(\bm{P})\Psi_{\bm{v} - \bm{e}_i}(\bm{P}).
\end{equation}
\label{net poly b}
\end{enumerate}
\end{prop}
\vspace{-1.5em}
\begin{remark}
When we say net polynomials are defined completely by the initial conditions given in \Cref{net poly}\ref{a values}, we mean by choosing appropriate initial values and substituting them into the definition of an elliptic net, one can obtain some other values, and hence define all values recursively. One can refer to \cite[Theorem 2.5]{Stange} for examples. 
\end{remark}

For the purpose of our paper, instead of just elliptic curves, we wish to define elliptic nets from any cubic Weierstrass curves. We will outline the relevant statements below, readers may refer to \cite[Section 5]{Stange} for details. 

\begin{defn}\label{appropriate}
Let $E/K$ be a cubic curve defined by (\ref{short EC}), not necessarily smooth. A set of points $\{P_1, \ldots, P_n\}$ on the non--singular part of $E$ is called \emph{appropriate} if the following hold:
\begin{enumerate}
\item $P_i\neq O$ for all $i=1, \ldots, n$;
\item $[2]P_i \neq O$ for all $i=1, \ldots, n$;
\item $P_i \neq \pm P_j$ for any $i \neq j$; and 
\item $[3]P_1\neq O$ whenever $n=1$.
\end{enumerate}

If we have an appropriate $n$--tuple of points $\bm{P}$ on the non--singular part of $C/K$, then we may define a map 
\[
\Psi_{\bm{v}}(\bm{P})\colon \Z^n \rightarrow K
\]
known as the elliptic net associated to $C$ and $\bm{P}$, which is an elliptic net. 
\end{defn}

\begin{prop}\cite[Corollary 5.2]{Stange}\label{net polynomial vanish}
For an elliptic net $\Psi_{\bm{v}}(\bm{P}) \colon \Z^n \rightarrow K$ associated to a cubic Weierstrass curve $E$ and appropriate points $\bm{P}$ on the non--singular part of $E/K$, we have $\Psi_{\bm{v}}(\bm{P})=0$ if and only if $\bm{v}\cdot \bm{P}=O$ on $E_{\text{ns}}$, where $E_{\text{ns}}$ denotes the set/group of non--singular points of $C$ defined over $K$ and $\bm{v}\cdot \bm{P}$ is as defined in \Cref{net poly def}. 
\end{prop}

\subsection{Resultants}
For the convenience of readers, we recall some basic facts and properties of resultants of polynomials. Throughout this section, let $f, g \in K[x]$ be polynomials of positive degree, in the form 
\begin{align*}
f(x)&=a_0x^n+\ldots+a_n, a_0 \neq 0, \\
g(x)&=b_0x^m+\ldots + b_m, b_0 \neq 0.
\end{align*}

\begin{defn}
The Sylvester matrix of $f$ and $g$ with respect to $x$, denoted $\Syl_{x}(f, g)$ is the following $(m+n)\times (m+n)$ matrix:
\[
\Syl_{x}(f, g)=
\begin{pmatrix}
a_0 & 0 & \cdots & 0 & b_0 & 0 & \cdots & 0\\
a_1 & a_0 & \cdots & 0 & b_1 & b_0 & \cdots & 0\\
a_2 & a_1 & \ddots & 0 & b_2 & b_1 & \ddots & 0\\
\vdots & \vdots & \ddots & a_0 & \vdots & \vdots & \ddots & b_0\\
a_n & a_{n-1} & \cdots & a_0 & b_m & b_{m-1} & \cdots & b_0\\
0 & a_n & \cdots & a_1 & 0 & b_m & \cdots & b_1\\
\vdots & \vdots & \ddots & \vdots & \vdots & \vdots & \ddots & \vdots\\
0 & 0 & \cdots & a_n & 0 & 0 & \cdots & b_m
\end{pmatrix}.
\]
The resultant of $f$ and $g$, with respect to $x$, denoted $\Res_{x}(f, g)$, is the determinant of the Sylvester matrix. Thus, 
\[
\Res_{x}(f, g)=\Det\left(\Syl_{x}(f, g)\right).
\]
\end{defn}

\begin{prop}
Let $f(x)$ and $g(x)$ be polynomials in $K[x]$ of positive degree. 
\begin{enumerate}
\item The resultant $\Res_{x}(f, g)$ is an integer polynomial in the coefficients of $f$ and $g$. Furthermore, $f$ and $g$ have a common factor in $K[x]$ if and only if $\Res_{x}(f, g)=0$. 
\item There are polynomials $A, B \in K[x]$ such that 
\[
Af+Bg=\Res_{x}(f, g).
\]
Furthermore, the coefficients of $A$ and $B$ are integer polynomials in the coefficients of $f$ and $g$. 
\end{enumerate}
\end{prop}
\begin{proof}
See \cite[Chapter 3, \S 6, Proposition 3, 5]{CoxAndLittle1}. 
\end{proof}

Below we quote several useful facts about resultants.

\begin{prop}\label{resultant properties}
Let $f(x)$ and $g(x)$ be polynomials in $K[x]$ of positive degree. 
\begin{enumerate}
\item  If $f(x)$ and $g(x)$ split into linear factors
\[
f(x)=a_0\prod_{i=1}^{n} (x-\lambda_i),\quad g(x)=b_0\prod_{j=1}^m(x-\mu_j),
\]
then
\[
\Res_{x}(f, g)=a_0^mb_0^m\prod_{i=1}^n\prod_{j=1}^m (\lambda_i-\mu_j)=a_0^m\prod_{i=1}^ng(\lambda_i)=(-1)^{mn}b_0^{n}\prod_{j=1}^{m}f(\mu_j).
\]
\item If $a_i$ is assigned with weight $n-i$ and $b_j$ is assigned with weight $m-j$, then $\Res_{x}(f, g)$ is homogenous of weight $mn$. 
\end{enumerate}
\end{prop}
\begin{proof}
See \cite[Chapter 3, \S 1, Exercise 7, 8, 10]{CoxAndLittle2}
\end{proof}

%% file: Elliptic_nets_and_choice_of_points.tex
\section{Criteria for net polynomials to vanish simultaneously}
In this section, we want to prove the main criterion for when the net polynomials $\Phi_{\bm{v}}(\bm{P})$ and $\Psi_{\bm{v}}(\bm{P})$ can vanish simultaneously. To begin with, we first quote an immediate consequence of (\ref{Ayad resultant}) from \cite[Exercise 3.7(c)]{Silverman1}.

\begin{prop}[Silverman]
Let $\Delta$ be the discriminant of the cubic Weierstrass curve $E$, defined by the equation (\ref{short EC}). If $\Delta \neq 0$, $\phi_n(x)$ and $\psi_n(x)^2$ are relatively prime polynomials in $K[x]$. 
\end{prop}

Both (\ref{Ayad resultant}) and the proposition above suggest that if the division polynomials are associated to a singular cubic Weierstrass curve, then they could share common zeros as elements of $\Z[x]$. Since net polynomials are a higher rank generalisation of division polynomials, we expect that a similar result should hold for net polynomials as well. Hence, we prove the result below. 

\begin{prop}\label{singular vanish}
Let $E/K \colon y^2=x^3+Ax+B$ be a singular cubic Weierstrass curve defined over $K$ and let $\bm{P}=(P_1, P_2) \in E\times E$ be a pair of distinct points. Let $\bm{v}=(v_1, v_2)\in \Z^2$. Then $\Phi_{\bm{v}}(\bm{P})$ and $\Psi_{\bm{v}}(\bm{P})$ vanish simultaneously if one of the following holds:
\begin{enumerate}
\item $|v_1|<2$, $|v_2|\geq 2$, and $P_2$ is the singular point;
\item $|v_2|<2$, $|v_1|\geq 2$, and $P_1$ is the singular point;
\item $|v_1|,|v_2|\geq 2$, and one of $P_1,P_2$ is the singular point.
\end{enumerate}
\end{prop}

\begin{remark}
In the definition of net polynomials (see \Cref{net poly def}), it was stated that $\Phi_{\bm{v}}(\bm{P})$ and $\Psi_{\bm{v}}(\bm{P})$ are elements in the affine coordinate ring 
\[
R_{2}=K[x_i, y_i]_{1\leq i \leq 2}[(x_{i}-x_{j})^{-1}]/\langle f(x_{i}, y_{i})\rangle_{1\leq i \leq 2},
\]
where $f$ is the defining polynomial (\ref{short EC}) for the elliptic curve $E/K$. Therefore, we do not allow both points to be evaluated at the point singularity in this proposition. 
\end{remark}

\begin{remark}
The proposition implies that $\Phi_{\bm{v}}(\bm{P})$ and $\Psi_{\bm{v}}(\bm{P})$ cannot vanish simultaneously for $\bm{v}\in\{(0, \pm1), (\pm1, 0), (\pm1, \pm1), (\pm 1, \mp 1)\}$, the non-vanishing of $\Psi_{\bm{v}}(\bm{P})$ follows directly from the initial values. For the first three vectors, these values are $\pm1$, while $\Psi_{\pm(1,-1)}(\bm{P})=\pm(x_2-x_1)$ (see \cite[Theorem 2.5]{Stange} for the relevant recurrence). Since one of $P_1,P_2$ is the singular point $(x_*,0)$, the equality $x_1=x_2$ would force both points to equal $(x_*,0)$, which we have excluded from this proposition. Therefore, $\Phi_{\bm{v}}(\bm{P})$ and $\Psi_{\bm{v}}(\bm{P})$ cannot vanish simultaneously at these vectors. 
\end{remark}
The main intuition behind the proof is that the group law on $E$ does not apply to the singular point. Since the net polynomials $\Psi_{\bm{v}}\colon \Z^2 \rightarrow K$ are used to consider $\Z$--linear combination of points, the point of singularity should not be a well-defined input. 
\begin{proof}
We have assumed that $E$ is singular, so let $P_*=(x_*,y_*)$ denote its singular point. Then the following statements hold:
\begin{enumerate}
    \begin{multicols}{2}
        \item $\Delta=4A^3+27B^2=0$;
        \item $\frac{\partial f}{\partial y}(P_*)=2y_*=0$, so $y_*=0$;
        \columnbreak
        \item $\frac{\partial f}{\partial x}(P_*)=3x_*^2+A=0$;
        \item $x_*^3-3x_*^3+B=0$, so $B=2x_*^3$.
    \end{multicols}
\end{enumerate}
This also means that the singular point is $P_*=(x_*, 0)$. 

\textcolor{red}{The structure of the proof is as follows: we first prove the asserted vanishing of $\Psi_{\bm v}$ by induction on the sup-norm
\[
    \lVert\bm v\rVert_\infty=\max\left(|v_1|,|v_2|\right).
\]
By \cite[Lemma 2.4]{Stange}, the values of $\Psi_{\bm{v}}$ are integrally-implied by $\Psi_{\bm{v}}$ with $\lVert\bm v\rVert_\infty \leq 4$, so we handle the base case $\bm v=n\bm e_i$ and $\lVert\bm v\rVert_\infty\leq 4$. These cases follow from the explicit formulae for net polynomials in \cite[Proposition~3.8]{Stange}, together with recurrences stated in the proof of \cite[Theorem~2.5]{Stange}. Next, we will use the recurrence used in the proof of \cite[Lemma~2.4]{Stange} to show that for $\lVert\bm v\rVert_\infty>4$, $\Psi_{\bm{v}}$ can be expressed as a monomial in net polynomials $\Psi_{\bm{v}}$ with smaller sup-norm, hence the inductive hypothesis can be applied and completing the induction. Finally, we deduce the corresponding vanishing of $\Phi_{\bm v}$ from the relation between $\Phi_{\bm v}$ using \Cref{net poly}(e).}

\textcolor{red}{
For $\bm{v}=n\bm{e}_i$, $2\leq n \leq 4$, this reduces to the rank $1$ case by \Cref{net poly}(c). Note that $E$ in this case is not an elliptic curve, so  \Cref{Ayad TFAE} is not applicable. However, using the initial conditions of $\Psi_{n\bm{e}_i}(\bm{P})=\psi_n(P_i)$ (say, \cite[Section 1]{Akbary}), we can easily see that 
\[
\psi_2(P_*)=\psi_3(P_*)=\psi_4(P_*)=0.
\]
Next, we use the explicit formulas for net polynomials given in \cite[Proposition 3.8]{Stange} to show that the net polynomials $\Psi_{\bm{v}}$ vanish when one of $P_1$ or $P_2$ equals $P_*$ for $\bm{v}=n\bm{e}_i$ for $\bm{v}=(1,2) \text{ and } (2,1)$.}

 For $\bm{v}=(1, 2)$, we have 
\begin{align*}
\Psi_{(1, 2)}(\bm{P})
&= \frac{(2x_2+x_1)(x_2-x_1)^2-(y_2-y_1)^2}{(x_2-x_1)^2} \\
&= \frac{-4x_1x_2^2+2x_2^3+x_1^3+x_2^2x_1-(y_1^2+y_2^2-2y_1y_2)}{(x_2-x_1)^2} \\
&=\frac{-3x_2x_1^2+2x_2^3-Ax_1-2B+2y_1y_2}{(x_2-x_1)^2},
\end{align*}
where in the last step, we have substituted $y_i^2=x_i^3+Ax_i+B$ using the definition of net polynomials. Note that if we choose $P_2=(x_2, y_2)=(x_*, 0)=P_*$ and recall that this implies $B=2x_*^3$ and $A=-3x_*^2$, we get 
\begin{align*}
\Psi_{(1, 2)}(\bm{P})
&=\frac{-3x_1x_*^2+x_*^3-Ax_1-Ax_*-2B}{(x_*-x_1)^2} \\
&=\frac{x_*^3-Ax_*-2B}{(x_*-x_1)^2} \\
&=\frac{x_*^3+3x_*^3-4x_*^3}{(x_*-x_1)^2} \\
&=0
\end{align*}

Although, note that if we choose $P_1=(x_1, y_1)=(x_*, 0)$ instead, then 
\begin{align*}
\Psi_{(1, 2)}(\bm{P})
&=\frac{-3x_*x_2^2+x_2^3-Ax_2-Ax_*-2B}{(x_*-x_2)^2} \\
&=\frac{-3x_*x_2^2+x_2^3+3x_*^2x_2-x_*^3}{(x_*-x_2)^2} \\
&=x_2-x_* \\
&=\Psi_{(1, -1)}(P_*, P_2).
\end{align*}

Similarly, for $\Psi_{(2, 1)}(\bm{P})$, we have 
\[
\Psi_{(2, 1)}(P_*, P_2)=0, \text{ and } \Psi_{(2, 1)}(P_1, P_*)=x_1-x_*=\Psi_{(-1, 1)}(P_1, P_*).
\]

For $\Psi_{(2,2)},\Psi_{(2, -1)}, \Psi_{(-1,2)}$ and $\Psi_{(2,-2)}$, we use the recurrences from the proof of \cite[Theorem 2.5]{Stange}, which tells us that
\begin{align*}
\Psi_{(2,2)}
&=
\frac{
\Psi_{(0,2)}\Psi_{(2,1)}
-
\Psi_{(2,0)}\Psi_{(1,2)}
}{
\Psi_{(1,-1)}
}, \\
\Psi_{(2,-1)}
&=
\Psi_{(2,0)}\Psi_{(0,2)}
-
\Psi_{(2,1)}\Psi_{(1,-1)}^2, \\
\Psi_{(-1,2)}
&=
\Psi_{(2,-1)}
-
\Psi_{(1,-1)}^3,
\text{ and }\\
\Psi_{(2,-2)}
&=
\Psi_{(0,2)}\Psi_{(2,-1)}\Psi_{(1,-1)}
-
\Psi_{(2,0)}\Psi_{(1,-2)}\Psi_{(1,-1)}.
\end{align*}

Note that
\[
\Psi_{(2,0)}(P_*,P_2)
=
\Psi_{(2,1)}(P_*,P_2)
=
\Psi_{(0,2)}(P_1,P_*)
=
\Psi_{(1,2)}(P_1,P_*)
=
0,
\]
so we have
\[
\Psi_{(2,2)}(P_*,P_2)
=
\Psi_{(2,2)}(P_1,P_*)
=
0.
\]

Next, from the result of $\Psi_{(0,2)}, \Psi_{(2,0)}$ and $\Psi_{(2,1)}$, we may conclude that
\[
\Psi_{(2,-1)}(P_*,P_2)=0,
\qquad
\Psi_{(2,-1)}(P_1,P_*)=-(x_1-x_*)^3.
\]

Similarly, we also have
\[
\Psi_{(1,-2)}(P_1,P_*)=0,
\qquad
\Psi_{(1,-2)}(P_*,P_2)=-(x_2-x_*)^3,
\]
and
\[
\Psi_{(2,-2)}(P_*,P_2)
=
\Psi_{(2,-2)}(P_1,P_*)
=
0.
\]

\textcolor{red}{For the remaining vectors with $\lVert\bm v\rVert_\infty \leq 4$, we once again use the recurrences from the proof of \cite[Theorem 2.5]{Stange}. For clarity, we list all of them explicitly. 
\begin{enumerate}
\item $\Psi_{(1,3)}=\Psi_{(2,2)}\Psi_{(0,2)}+\Psi_{(1,-1)}\Psi_{(1,2)}^2$.
\item $\Psi_{(1,4)}=\Psi_{(1,3)}\Psi_{(0,2)}^2-\Psi_{(1,2)}^2\Psi_{(0,3)}$.
\item $\Psi_{(-1, 3)}=\Psi_{(1, -1)}^2\Psi_{(1, 2)}\Psi_{(-1, 2)}-\Psi_{(0, 2)}^2\Psi_{(-1, 2)}$.
\item $\Psi_{(-1, 4)}=\Psi_{(-1, 2)}^2\Psi_{(0, 3)}-\Psi_{(-1, 3)}\Psi_{(0, 2)}^2\Psi_{(-1, 1)}$. 
\item $\Psi_{(2,3)}=\Psi_{(2,2)}\Psi_{(0,2)} -\Psi_{(2,1)}\Psi_{(1,2)}^2$.
\item $\Psi_{(2,4)}=\Psi_{(2,3)}\Psi_{(0,2)}\Psi_{(1,2)} -\Psi_{(2,2)}\Psi_{(1,2)}\Psi_{(1,3)}$.
\item $\Psi_{(-2,3)}=\Psi_{(2,-1)}\Psi_{(-1,2)}^2 -\Psi_{(2,-2)}\Psi_{(0,2)}\Psi_{(1,-1)}^2$. 
\item $\Psi_{(-2,4)}=\Psi_{(2,-2)}\Psi_{(0,3)}\Psi_{(-1,2)}\Psi_{(1,-1)}+\Psi_{(2,-1)}\Psi_{(0,2)}\Psi_{(-1,3)}\Psi_{(-1,2)}$.
\item $\Psi_{(3, 1)}=\Psi_{(2, 2)}\Psi_{(2, 0)}-\Psi_{(1, -1)}\Psi_{(2, 1)}^2$.
\item $\Psi_{(3, 2)}=\Psi_{(2, 2)}\Psi_{(2, 0)}-\Psi_{(1, 2)}\Psi_{(2, 1)}^2$.
\item $\Psi_{(3, 3)}=\Psi_{(3, 2)}\Psi_{(1, 2)}-\Psi_{(2, 1)}\Psi_{(2, 2)}^2$.
\item $\Psi_{(3,4)}=\Psi_{(2,3)}\Psi_{(2,1)}\Psi_{(1,2)}^2-\Psi_{(1,3)}\Psi_{(2,2)}^2$.
\item $\Psi_{(3, -1)}=\Psi_{(1, -1)}^2\Psi_{(2, 1)}\Psi_{(2, -1)}-\Psi_{(2, 0)}^2\Psi_{(1, -2)}$.
\item $\Psi_{(3,-2)}=\Psi_{(1, -1)}^2\Psi_{(2,-2)}\Psi_{(2,0)}-\Psi_{(1,-2)}\Psi_{(2,-1)}^2$.
\item $\Psi_{(3, -3)}=\Psi_{(3, -1)}\Psi_{(1, -1)}\Psi_{(1, -2)}+\Psi_{(2, -2)}\Psi_{(0, 2)}\Psi_{(2, -1)}^2$.
\item $\Psi_{(3,-4)}=\Psi_{(2,-3)}\Psi_{(2,-1)}\Psi_{(1,-2)}^2-\Psi_{(1, -1)}\Psi_{(1,-3)}\Psi_{(2,-2)}^2$.
\item $\Psi_{(4, 1)}=\Psi_{(3,1)}\Psi_{(2,0)}^2-\Psi_{(2,1)}^{2}\Psi_{(3,0)}$.
\item $\Psi_{(4, 2)}=\Psi_{(3,2)}\Psi_{(2,0)}\Psi_{(2,1)}-\Psi_{(2,2)}\Psi_{(2,1)}\Psi_{(3,1)}$.
\item $\Psi_{(4, 3)}=\Psi_{(3,2)}\Psi_{(1,2)}\Psi_{(2,1)}^{2}-\Psi_{(3,1)}\Psi_{(2,2)}^{2}$.
\item $\Psi_{(4, 4)}=\Psi_{(3,3)}\Psi_{(2,1)}\Psi_{(1,2)}\Psi_{(2,2)}-\Psi_{(2,3)}\Psi_{(2,2)}\Psi_{(3,2)}$.
\item $\Psi_{(4, -1)}=\Psi_{(2, -1)}^2\Psi_{(3, 0)}-\Psi_{(3, -1)}\Psi_{(2, 0)}^2\Psi_{(1, -1)}$. 
\item $\Psi_{(4, -2)}=\Psi_{(2,-2)}\Psi_{(3,0)}\Psi_{(2,-1)}\Psi_{(1,-1)}-\Psi_{(1,-2)}\Psi_{(2,0)}\Psi_{(3,-1)}\Psi_{(2,-1)}$.
\item $\Psi_{(4, -3)}=\Psi_{(1, -1)}\Psi_{(3, -1)}\Psi_{(2, -2)}^2-\Psi_{(3, -2)}\Psi_{(1, -2)}\Psi_{(2, -1)}^2$. 
\item $\Psi_{(4, -4)}=\Psi_{(2,-3)}\Psi_{(3,-1)}\Psi_{(2,-2)}\Psi_{(1,-2)}-\Psi_{(1,-3)}\Psi_{(2,-1)}\Psi_{(3, -2)}\Psi_{(2, -2)}$.
\end{enumerate}
}
By \Cref{odd net}, $\Psi_{-\bm{v}}=-\Psi_{\bm{v}}$, so the asserted vanishing at $\bm{v}$ also holds at $-\bm{v}$. Hence, it suffices to consider one representative of each pair $\{\bm{v},-\bm{v}\}$. Together with the vectors already considered, the above recurrences and formulae cover all vectors $\bm{v}$ with $2\leq\lVert\bm{v}\rVert_\infty\leq4$ up to sign and we obtain
\begin{enumerate}
\item $\Psi_{(3, 1)}(P_*,P_2)=\Psi_{(3, -1)}(P_*,P_2)=\Psi_{(4, 1)}(P_*,P_2)=\Psi_{(4,-1)}(P_*,P_2)=0$,
\item $\Psi_{(1, 3)}(P_1,P_*)=\Psi_{(-1, 3)}(P_1,P_*)=\Psi_{(1, 4)}(P_1,P_*)=\Psi_{(-1, 4)}(P_1,P_*)=0$, \text{and }
\item $\Psi_{\bm{v}}(P_1,P_*)=\Psi_{\bm{v}}(P_*,P_2)=0$ for $\bm{v}=(v_1, v_2)$ with $2\leq |v_{i}| \leq 4$. 
\end{enumerate}

We now follow the proof of \cite[Lemma 2.4]{Stange} and induct on $\lVert\bm v\rVert_\infty$. The assertion has already been shown for all relevant vectors $\bm{v} \in \Z^2$ with $\lVert\bm{v}\rVert_\infty\leq4$. Let $N_0>4$, and suppose that the assertion holds for every $\bm u\in\mathbb Z^2$ satisfying $\lVert\bm u\rVert_\infty<N_0$. 

Let $\bm{v}\in \Z^2$ satisfy $\lVert\bm v\rVert_\infty=N_0$. For $i=1, 2$, define $w_i=\lceil \frac{v_i}{2}\rceil$. Stange showed that we can construct a recurrence, depending on the parity of $v_i$, demonstrating that $\Psi_{\bm{v}}$ can be expressed in terms of $\Psi_{\bm{u}}$ with $\lVert\bm{u}\rvert_{\infty} \leq N_0$. 

\underline{Case I:} without loss of generality, assume $v_1$ is odd and $v_2$ is even.
\[
\Psi_{(v_1, v_2)}=\Psi_{(w_1, w_2+1)}\Psi_{(w_1, w_2-1)}\Psi_{(w_1-1, w_2)}^2-\Psi_{(w_1-1, w_2+1)}\Psi_{(w_1-1, w_2-1)}\Psi_{(w_1, w_2)}^2.
\]
\underline{Case II:} both $v_1$ and $v_2$ are odd.
\[
\Psi_{(v_1, v_2)}=\Psi_{(w_1, w_2+1)}\Psi_{(w_1, w_2-1)}\Psi_{(w_1-1, w_2-1)}^2-\Psi_{(w_1-1, w_2)}\Psi_{(w_1-1, w_2-2)}\Psi_{(w_1, w_2)}^2.
\]
\underline{Case III:} both $v_1$ and $v_2$ are even.
\[
\Psi_{(v_1, v_2)}=\Psi_{(w_1+1, w_2+1)}\Psi_{(w_1, w_2-1)}\Psi_{(w_1-1, w_2)}\Psi_{(w_1, w_2)}-\Psi_{(w_1, w_2+1)}\Psi_{(w_1-1, w_2-1)}\Psi_{(w_1+1, w_2)}\Psi_{(w_1, w_2)}.
\]
If $v_i$ is even, every $i$-th coordinate appearing on the right-hand side belongs to $\{w_i-1,w_i,w_i+1\}$. Hence
\[
    |u_i|\leq |w_i|+1=\frac{|v_i|+2}{2}.
\]
If $v_i$ is odd, every such coordinate belongs to $\{w_i-2,w_i-1,w_i,w_i+1\}$. Since $w_i=(v_i+1)/2$, we may deduce that
\[
    |u_i|
    \leq \max\left(|w_i-2|,|w_i+1|\right)
    =\frac{|v_i|+3}{2}.
\]
Consequently, every vector $\bm{u}$ appearing on the right-hand side of the three recurrences satisfies
\[
    \lVert\bm{u}\rVert_\infty
    \leq \frac{\lVert\bm{v}\rVert_\infty+3}{2}
    =\frac{N_0+3}{2}<N_0,
\]
where the final inequality follows from $N_0>4$.
Suppose that $P_i=P_*$ and $|v_i|\geq 2$. We claim that each of the two monomials contains a factor $\Psi_{\bm{u}}$ with $|u_i|\geq 2$.

If $v_i$ is even, then $|w_i|\geq 1$ and each monomial contains factors whose $i$-th coordinates are $w_i-1$ and $w_i+1$. At least one of these coordinates has absolute value at least $2$. If $v_i$ is odd, then $|v_i|\geq 3$. In Cases I and II, the relevant coordinates are $w_i$ and $w_i-1$, at least one of which has absolute value at least $2$, except when $i=2$ in Case II. In that case, the first monomial contains coordinates $w_2+1$ and $w_2-1$, while the second contains coordinates $w_2$ and $w_2-2$. Since $|v_2|\geq 3$, each pair contains a coordinate of absolute value at least $2$.

Overall, each monomial contains a factor $\Psi_{\bm{u}}$ such that $\lVert\bm{u} \rVert_\infty<N_0$ and $|u_i|\geq 2$. By the inductive hypothesis, this factor vanishes at
$\bm{P}$. Both monomials therefore vanish, giving $\Psi_{\bm{v}}(\bm{P})=0$ and completing the induction.

Finally, we turn our attention to the numerator function $\Phi_{\bm{v}}$. From \Cref{net poly}, we can express $\Phi_{\bm{v}}$ in terms of $\Psi_{\bm{v}}$:
\begin{equation}\label{phi}
\Phi_{\bm{v}}=\Psi_{\bm{v}}^2x(P_i)-\Psi_{\bm{v}+\bm{e}_i}\Psi_{\bm{v}-\bm{e}_i},
\end{equation}
which allows us to conclude the proof:
\begin{enumerate}
\item If $\bm{v}=(v_1, \pm 1)$ with $|v_1|>1$, then choose $i=2$ in (\ref{phi}), which gives us 
\begin{align*}
\Phi_{\bm{v}}(P_*, P_2)
&=x_2\Psi_{\bm{v}}(P_*, P_2)^2-\Psi_{(v_1, v_2+1)}(P_*, P_2)\Psi_{(v_1, v_2-1)}(P_*, P_2) \\
&=0 \\
&=\Psi_{\bm{v}}(P_*, P_2).
\end{align*}
\item If $\bm{v}=(\pm1, v_2)$ with $|v_2|>1$, then choose $i=1$ in (\ref{phi}), similar to the case above, it also gives us 
\[
\Phi_{\bm{v}}(P_1, P_*)=\Psi_{\bm{v}}(P_1, P_*)=0.
\]
\item If $\bm{v}$ satisfies $\min_{i=1, 2}|v_i|\geq 2$, by the previous argument for $\Psi_{\bm{v}}$, we obtain
\[
\Phi_{\bm{v}}(P_1, P_*)=\Phi_{\bm{v}}(P_*, P_2)=\Psi_{\bm{v}}(P_*, P_2)=\Psi_{\bm{v}}(P_1, P_*)=0.
\]
\end{enumerate}
In all cases, $\Phi_{\bm{v}}$ and $\Psi_{\bm{v}}$ vanish simultaneously when one of $P_1$ and $P_2$ is evaluated at the singular point $P_*$. 
\end{proof}

Now we have shown that if a cubic Weierstrass curve $E$ is singular, then the net polynomials associated to $E$ have a common vanishing locus in $E^2$. We would like to show that this is the only condition that the net polynomials associated to $E$ and an appropriate pair of points $\bm{P} \in E(K)$ can vanish simultaneously. 

\begin{prop}
Let $\Psi_{\bm{v}}(\bm{P})\colon \Z^2\rightarrow K$ be an elliptic net associared to a cubic Weierstrass curve $E$ and $\bm{P} \in E(K)^2$ be a pair of appropriate, non--singular points, then $\Psi_{\bm{v}}(\bm{P})$ and $\Phi_{\bm{v}}(\bm{P})$ do not both vanish. 
\end{prop}
\begin{proof}
Recall from \Cref{net polynomial vanish} that under our assumptions, $\Psi_{\bm{v}}(\bm{P})=0$ if and only if 
\[
\bm{v}\cdot \bm{P}=[v_1]P_1+[v_2]P_2=O.
\]
Therefore, it remains to show that $\Phi_{\bm{v}}(\bm{P})\neq 0$ when $\bm{v}\cdot \bm{P}=[v_1]P_1+[v_2]P_2=O$. Assume this is the case, then from \Cref{net poly}\ref{net poly b}, we obtain 
\[
\Phi_{\bm{v}}(\bm{P})= \Psi_{\bm{v}}^{2}(\bm{P})\,x(P_i)- \Psi_{\bm{v} + \bm{e}_i}(\bm{P})\Psi_{\bm{v} - \bm{e}_i}(\bm{P})=\Psi_{\bm{v} + \bm{e}_i}(\bm{P})\Psi_{\bm{v} - \bm{e}_i}(\bm{P}).
\]
Since $P_1, P_2$ are appropriate, neither of them are the point of infinity. Therefore, $\left(\bm{v}\pm\bm{e}_i\right)\cdot \bm{P} = \pm P_i \neq O$ and $\Phi_{\bm{v}}(\bm{P})$ is necessarily nonzero. 
\end{proof}

\begin{example}
Consider the curve $E/K\colon y^2=x^3-3x+2$, a cubic curve in short Weierstrass form defined over a number field $K$ with singularity at $P_*=(1, 0)$. Therefore, $E(K)\setminus\{(1, 0)\}$ forms a group on the curve. Below we list a few values of $\Psi_{\bm{v}}$ and $\Phi_{\bm{v}}$,

\begin{itemize}
\item  For $\bm{v}=(0, 2)$, we have $\Psi_{(0, 2)}=4(x_2-1)^2(x+2)$ and $\Phi_{(0, 2)}=(x_2-1)^2(x_2^2+2x_2+9)$. They both vanish at $(x_2, y_2)=(1, 0)$. 
\item For $\bm{v}=(1, 2)$, we have the following:
\begin{align*}
\Psi_{(1, 2)}
&=\frac{-x_1^3 + 3x_1x_2^2 - 2x_2^3 + y_1^2 - 2y_1y_2 + y_2^2}{(x_2-x_1)^2};
\\[1ex]
\Phi_{(1, 2)}
&=
\frac{1}{(x_2-x_1)^4}
\Bigl(
x_1^7 - 6x_1^5x_2^2 + 4x_1^4x_2^3 + 9x_1^3x_2^4- 12x_1^2x_2^5 + 4x_1x_2^6+ 4x_1^5y_1y_2 - 8x_1^4y_1x_2y_2
\\
&\quad
- 8x_1^3y_1x_2^2y_2 + 32x_1^2y_1x_2^3y_2- 28x_1y_1x_2^4y_2 + 8y_1x_2^5y_2- 8x_1^5y_2^2 + 28x_1^4x_2y_2^2-32x_1^3x_2^2y_2^2 
\\
&\quad
+ 8x_1^2x_2^3y_2^2+ 8x_1x_2^4y_2^2- 4x_2^5y_2^2- 2x_1^4y_1^2 + 6x_1^2y_1^2x_2^2- 4x_1y_1^2x_2^3 + 4x_1^4y_1y_2- 4x_1^2y_1^3y_2 
\\
&\quad
+ 8x_1y_1^3x_2y_2- 12x_1^2y_1x_2^2y_2- 4y_1^3x_2^2y_2 + 8x_1y_1x_2^3y_2- 2x_1^4y_2^2 + 12x_1^2y_1^2y_2^2- 24x_1y_1^2x_2y_2^2
\\
&\quad
+ 6x_1^2x_2^2y_2^2+ 12y_1^2x_2^2y_2^2- 4x_1x_2^3y_2^2 - 12x_1^2y_1y_2^3+ 24x_1y_1x_2y_2^3- 12y_1x_2^2y_2^3
\\
&\quad
+ 4x_1^2y_2^4- 8x_1x_2y_2^4 + 4x_2^2y_2^4+ x_1y_1^4 - 4x_1y_1^3y_2+ 6x_1y_1^2y_2^2 - 4x_1y_1y_2^3 + x_1y_2^4
\Bigr).
\end{align*}
If we replace $y_i^2$ with $x_i^3-3x+2$ for $i=1, 2$, we get 
\begin{align*}
\Psi_{(1, 2)}
&=\frac{(x_2-1)(3x_1x_2+3x_1-x_2^2-x_2-4)-2y_1y_2}{(x_2-x_1)^2}; \\
\Phi_{(1,2)}
&=\frac{-4y_1y_2(x_2-1)}{(x_2-x_1)^4}\Bigl(3x_1^3x_2 + 3x_1^3 - 5x_1^2x_2^2 - 2x_1^2x_2- 5x_1^2 + x_1x_2^3 + 15x_1x_2 
\\
&\quad 
- 4x_1 + x_2^4 + x_2^3 - 8x_2^2\Bigr) 
\\
&\quad +\left(\frac{x_2-1}{x_2-x_1}\right)^2\Bigl(4x_1^5x_2 + 8x_1^5 + 4x_1^4x_2^2 + 12x_1^4x_2 + 8x_1^4 - 20x_1^3x_2^3 - 31x_1^3x_2^2 
\\
&\quad
- 18x_1^3x_2 -63x_1^3 + 12x_1^2x_2^4 + 18x_1^2x_2^3 + 48x_1^2x_2^2 + 110x_1^2x_2 + 16x_1^2 
\\
&\quad
+ x_1x_2^4 - 10x_1x_2^3 - 79x_1x_2^2 - 112x_1x_2 + 32x_1 - 12x_2^4 + 8x_2^3 + 64x_2^2\Bigr). 
\end{align*}
They both vanish at $(x_2, y_2)=(1, 0)$, but not at $(x_1, y_1)=(1, 0)$ as stated by the proposition. 
\item For $\bm{v}=(2, -2)$, we provide the values of $\Psi_{(2, -2)}$ and $\Phi_{(2, -2)}$ with $y_i^2$ replaced by $x_i^3-3x+2$ for $i=1, 2$. 
	\begin{align*}
\Psi_{(2, -2)}
&=
(x_2-x_1)^5\left[y_1(x_2 - 1)(x_2^2+ 3x_1x_2 + x_2 + 3x_1 - 8)\right]
\\
&\quad
+(x_2-x_1)^5\left[y_2(x_1-1)(x_1^2 + 3x_1x_2 + x_1 + 3x_2 - 8)\right] 
\\
\Phi_{(2, -2)}
&=
8 y_1 y_2(x_2-x_1)^{16}\bigl(x_1^5 x_2^2 + 7 x_1^4 x_2^3 + 7 x_1^3 x_2^4 + x_1^2 x_2^5+ x_1^5 x_2 + 8 x_1^4 x_2^2 + 14 x_1^3 x_2^3 
\\
&\quad
+ 8 x_1^2 x_2^4 + x_1 x_2^5+ 4 x_1^5 - 22 x_1^4 x_2 - 30 x_1^3 x_2^2 - 30 x_1^2 x_2^3 - 22 x_1 x_2^4 + 4 x_2^5+ 13 x_1^4 
\\
&\quad
- 100 x_1^3 x_2 + 46 x_1^2 x_2^2 - 100 x_1 x_2^3 + 13 x_2^4+ 25 x_1^3 + 111 x_1^2 x_2 + 111 x_1 x_2^2 + 25 x_2^3+ 20 x_1^2 
\\
&\quad
+ 56 x_1 x_2  + 20 x_2^2 - 160 x_1 - 160 x_2 + 128\bigr)
\\
&\quad 
+(x_2-x_1)^{16}(x_1-1)(x_2-1)\bigl(x_1^6 x_2^2 + 28 x_1^5 x_2^3 + 70 x_1^4 x_2^4 + 28 x_1^3 x_2^5 + x_1^2 x_2^6+ 2 x_1^6 x_2 
\\
&\quad
+ 58 x_1^5 x_2^2 + 196 x_1^4 x_2^3 + 196 x_1^3 x_2^4 + 58 x_1^2 x_2^5 + 2 x_1 x_2^6+ 9 x_1^6 + 16 x_1^5 x_2 - 209 x_1^4 x_2^2 
\\
&\quad
+ 112 x_1^3 x_2^3 - 209 x_1^2 x_2^4 + 16 x_1 x_2^5 + 9 x_2^6+ 114 x_1^5 - 726 x_1^4 x_2 - 668 x_1^3 x_2^2 - 668 x_1^2 x_2^3 
\\
&\quad
- 726 x_1 x_2^4 + 114 x_2^5+ 273 x_1^4 - 1060 x_1^3 x_2 + 1702 x_1^2 x_2^2 - 1060 x_1 x_2^3 + 273 x_2^4+ 576 x_1^3 
\\
&\quad
+ 1984 x_1^2 x_2 + 1984 x_1 x_2^2 + 576 x_2^3- 384 x_1^2 - 256 x_1 x_2 - 384 x_2^2 - 2048 x_1 - 2048 x_2 + 2048
\bigr)
	\end{align*}
They both vanish whenever we evaluate $(x_1, y_1)=(1, 0)$ and $(x_2, y_2)=(1, 0)$ as stated by the proposition. 
\end{itemize}
\end{example}

%% file: Resultant_and_net_polynomials.tex
\section{Resultant and net polynomials}\label{Section 4}
From the previous section, we now know that if the net polynomials are associated to a singular cubic Weierstrass curve, then they could possibly both vanish. In the case of rank $1$ division polynomials, this result is implied by (\ref{Ayad resultant}) using the method of resultants. As stated in \cite[Exercise III.3.7(a)]{Silverman1}, the division polynomial $\psi_n$ satisfies
\[
\psi_n(x)\in\begin{cases}
			\Z[x, A, B], & \text{$n$ odd}\\
            2y\Z[x, A, B], & \text{$n$ even}
		 \end{cases}.
\]
Therefore, we always consider $\psi_n^2$ instead to eliminate the $y$ variable using the Weierstrass equation, which is also equivalent to 
\[
\psi_n^2(x)=\Res_{y}(\psi_n(x, y), y^2-x^3-Ax-B).
\]
Therefore, when we write $\Res(\phi_n, \psi_n^2)$, we are actually considering 
\[
\Res_{x}(\Res_y(\psi_n(x, y), y^2-x^3-Ax-B), \phi_n).
\]
We would like to do the same for net polynomials, eliminate all the $y_i$--variables such that both $\Psi_{\bm{v}}(\bm{P})$ and $\Phi_{\bm{v}}(\bm{P})$ can be defined just by $x_1=x(P_1)$ and $x_2=x(P_2)$ using the same method. However, as we can see from the definition of net polynomials, they are not truly polynomials. Our solution is to multiply it by powers of $(x_2-x_1)$, then the denominators will be eliminated and $\Psi_{\bm{v}}(\bm{P})$ and $\Phi_{\bm{v}}(\bm{P})$ will truly be polynomials in $\Z[A, B][x_1, x_2, y_1, y_2]$. 

According to the proof of  \cite[Proposition 7.1]{Stange}, for $E \colon y^2=x^3+Ax+B$,
\begin{equation}\label{homothety}
\Psi_{\bm{v}}(\lambda^2x_1; \lambda^2x_2; \lambda^3y_1, \lambda^3y_2, \lambda^4A, \lambda^6B)=\lambda^{g(\bm{v})}\Psi_{\bm{v}}(x_1, x_2, y_1, y_2, A, B),
\end{equation}
where $g(\bm{v})=v_1^2+v_2^2-v_1v_2-1$. This tells us that $\Psi_{\bm{v}}(\bm{P})$ is homogenous of weight $g(\bm{v})$. The maximum number of $(x_2-x_1)^{-1}$ terms that can show up in $\Psi_{\bm{v}}(\bm{P})$ is then $g(\bm{v})$ as well, since the term $(y_2-y_1)/(x_2-x_1)$ is homogenous of weight $1$. Next, note that for any $\lambda \in \R\setminus\{0\}$, (\ref{homothety}) tells us that
\[
x\left(\bm{v}\cdot\bm{P}\right)=\frac{\Phi_{\bm{v}}(\bm{P})}{\Psi_{\bm{v}}^2(\bm{P})}
\implies 
\lambda^{2}x\left(\bm{v}\cdot\bm{P}\right)=\frac{\lambda^{2g(\bm{v})+2}\Phi_{\bm{v}}(\bm{P})}{\lambda^{2g(\bm{v})}\Psi_{\bm{v}}^2(\bm{P})},
\]
which implies that $\Phi_{\bm{v}}(\bm{P})$ is homogenous of weight $2g(\bm{v})+2$. If we multiply both $\Psi_{\bm{v}}$ and $\Phi_{\bm{v}}$ by $(x_2-x_1)^{g(\bm{v})+1}$, we are guaranteed to clear the denominators. Therefore, if we denote 
\begin{equation} \label{new poly}
\begin{aligned}
\Psi_{\bm{v}}^{*}
&=(x_2-x_1)^{g(\bm{v})+1}\Psi_{\bm{v}}=(x_2-x_1)^{v_1^2+v_2^2-v_1v_2}\Psi_{\bm{v}}; \\
\Phi_{\bm{v}}^{*}
&=(x_2-x_1)^{2g(\bm{v})+2}\Phi_{\bm{v}}=(x_2-x_1)^{2(v_1^2+v_2^2-v_1v_2)}\Phi_{\bm{v}}, 
\end{aligned}
\end{equation}
then $\Psi_{\bm{v}}^{*}$ and $\Phi_{\bm{v}}^{*}$ are truly polynomials in $\Z[A, B][x_1, x_2, y_1, y_2]$ and we can apply resultants to eliminate the $y_i$--variables. Define 
\[
f_1(x_1, x_2, y_1, y_2)=y_1^2-x_1^3-Ax_1-B, \text{ and } f_2(x_1, x_2, y_1, y_2)=y_2^2-x_2^3-Ax_2-B,
\]
then \Cref{resultant properties} gives us
\begin{equation*}
\Res_{y_2}(\Psi_{\bm{v}}^{*2}, f_2)(x_1, x_2, y_1)=\Psi_{\bm{v}}^{*2}\left(x_1, x_2, y_1, \sqrt{x_2^3+Ax_2+B}\right)\Psi_{\bm{v}}^{*2}\left(x_1, x_2, y_1, -\sqrt{x_2^3+Ax_2+B}\right).
\end{equation*}
Next, we eliminate $y_1$ from $\Psi_{\bm{v}}^{*2}$ using resultants again, which gives us
\begin{align*}
&\quad\Res_{y_1}\left(\Res_{y_2}(\Psi_{\bm{v}}^{*2}, f_2), f_1\right) \\
&=\Res_{y_2}(\Psi_{\bm{v}}^{*2}, f_2)\left(x_1, x_2, \sqrt{x_1^3+Ax_1+B}\right)\Res_{y_2}(\Psi_{\bm{v}}^{*2}, f_2)\left(x_1, x_2, -\sqrt{x_1^3+Ax_1+B}\right) \\
&=\Psi_{\bm{v}}^{*2}\left(x_1, x_2,  \sqrt{x_1^3+Ax_1+B}, \sqrt{x_2^3+Ax_2+B}\right)\Psi_{\bm{v}}^{*2}\left(x_1, x_2,  -\sqrt{x_1^3+Ax_1+B}, \sqrt{x_2^3+Ax_2+B}\right) \\
&\quad \Psi_{\bm{v}}^{*2}\left(x_1, x_2,  \sqrt{x_1^3+Ax_1+B}, -\sqrt{x_2^3+Ax_2+B}\right)\Psi_{\bm{v}}^{*2}\left(x_1, x_2,  -\sqrt{x_1^3+Ax_1+B}, -\sqrt{x_2^3+Ax_2+B}\right)
\end{align*}

\begin{defn}\label{R1 R2}
Let $\Psi_{\bm{v}}^{*}$ and $\Phi_{\bm{v}}^{*}$ be as in (\ref{new poly}). We define 
\[
R_1(x_1, x_2)=\Res_{y_1}\left(\Res_{y_2}(\Psi_{\bm{v}}^{*2}, f_2), f_1\right) \text{and } R_2(x_1, x_2)=\Res_{y_1}\left(\Res_{y_2}(\Phi_{\bm{v}}^{*}-t\Psi_{\bm{v}}^{*2}, f_2), f_1\right),
\]
where $t$ is an element in $\overline{\Z}$, so both $R_1(x_1, x_2)$ and $R_2(x_1, x_2)$ are both elements of $\Z[A, B, t][x_1, x_2]/\langle f_1, f_2\rangle$.
\end{defn}

The reason why we consider $\Phi_{\bm{v}}^{*}-t\Psi_{\bm{v}}^{*2}$ instead of just $\Phi_{\bm{v}}^{*}$ is due to the technical reason below: when we take resultants to eliminate $y_1$ and $y_2$, $\Psi_{\bm{v}}^{*2}$ and $\Phi_{\bm{v}}$ are unable to distinguish between $P_i$ and $-P_i$. More specifically, if $P_1, P_2$ are appropriate points on the cubic curve $E$, then \Cref{net polynomial vanish} implies the following:
\begin{align*}
\Psi_{\bm{v}}^{*2}\left(x_1, x_2,  -\sqrt{x_1^3+Ax_1+B}, \sqrt{x_2^3+Ax_2+B}\right)=0 &\iff [-v_1]P_1+[v_2]P_2=O, \\
\Psi_{\bm{v}}^{*2}\left(x_1, x_2,  \sqrt{x_1^3+Ax_1+B}, -\sqrt{x_2^3+Ax_2+B}\right)=0 &\iff [v_1]P_1-[v_2]P_2=O, \\
\Phi_{\bm{v}}\left(x_1, x_2,  -\sqrt{x_1^3+Ax_1+B}, \sqrt{x_2^3+Ax_2+B}\right)=0 &\iff x\left([-v_1]P_1+[v_2]P_2\right)=t, \\
\Phi_{\bm{v}}\left(x_1, x_2,  \sqrt{x_1^3+Ax_1+B}, -\sqrt{x_2^3+Ax_2+B}\right)=0 &\iff x\left([v_1]P_1-[v_2]P_2\right)=t.
\end{align*}

This means that knowing when $R_1(x_1, x_2)$ and $\Res_{y_1}\left(\Res_{y_2}(\Phi_{\bm{v}}^{*2}, f_2), f_1\right)$ vanish simultaneously does not fully reflect when $\Psi_{\bm{v}}(\bm{P})$ and $\Phi_{\bm{v}}(\bm{P})$ vanish simultaneously. We summarise the above in the following lemma.
\begin{lemma}\label{zero condition}
Let $E/K$ be a cubic curve defined over a number field $K$, not necessarily smooth, defined by (\ref{short EC}). Let $R_1(x_1, x_2), R_2(x_1, x_2) \in \Z[A, B, t][x_1, x_2]$ be as in \Cref{R1 R2}, where $t \in \overline{\Z}$. Then $R_1(x_1, x_2)=R_2(x_1, x_2)=0$ if and only if there exists a pair of appropriate points $\bm{P}=(P_1, P_2) \in E(\overline{\Q})^2$ such that 
\[
[v_1]P_1\pm [v_2]P_2=O \text{ and } x\left([v_1]P_1\mp [v_2]P_2\right)=t.
\]
\end{lemma}

In the next section, when we take resultant of $R_1$ and $R_2$, we will choose a specific value for $t$ to eliminate this problem. Another technicality to be emphasised here is that we must assume $P_1\neq \pm P_2$, which is part of the definition of appropriate points (see \Cref{appropriate}). By construction, $\Psi_{\bm{v}}^{*2}$ and $\Phi_{\bm{v}}^{*}$ might contain factors of $(x_2-x_1)$. This also implies that if we take resultant of $R_1(x_1, x_2)$ and $R_2(x_1, x_2)$ with respect to $x_1$ and $x_2$, we are very likely to get $0$ trivially, which is not a vanishing condition we want. To avoid this problem, we perform change of variable
\[
x_1 \mapsto -X+x_2, \quad x_2 \mapsto x_2,
\]
then despite $R_1(X, x_2)$ and $R_2(X, x_2)$ being possibly divisible by powers of $X$, if we take resultant of $R_1(X, x_2)$ and $R_2(X, x_2)$ with respect to $x_2$, we no longer get trivial zeros.

%% file: Main_result.tex
\section{Proof of \Cref{main result}}
Throughout this section, unless specified or otherwise, we let $K$ be a number field and $\OO_K$ be its ring of integers. 
\subsection{Irreducible factors of the resultant}
To prove our result, we follow the strategy in \cite[Lemma 2]{AyadA}. That is, we wish to obtain a formula for $\Res_{x_2}(R_1, R_2)(X)$. We first prove that, modulo any odd prime, any common irreducible factor of $R_1(x_1, x_2)$ and $R_2(x_1, x_2)$ can only be $x_2-x_1$ possibly. 

\begin{lemma}\label{algebraic geometry}
For any odd prime $\theta$, $\Psi_{\bm{v}}^{*2}$ and $\Phi_{\bm{v}}^{*}$ do not share any common factor as elements in $\Z[x_1, x_2, y_1, y_2, A, B]/\langle \theta, f_1, f_2\rangle$, except $(x_2-x_1)$ possibly, where $f_i=y_i^2-x_i^3-Ax_i-B$ for $i=1, 2$. 
\end{lemma}
\begin{proof}
Suppose $F(x_1, x_2, y_1, y_2, A, B)\neq x_2-x_1$ is an irreducible common factor of $\Psi_{\bm{v}}^{*2}$ and $\Phi_{\bm{v}}^{*}$ in the quotient ring $\overline{\F_{\theta}}[x_1, x_2, y_1, y_2, A, B]/\langle f_1, f_2 \rangle$. If $P_1=(x_1, y_1), P_2=(x_2, y_2)$, then $f_1(P_1)=f_2(P_2)=0$.  Since we assumed $F$ is an irreducible common factor not associated to $x_2-x_1$, we may choose a point $(x_1, x_2, y_1, y_2, A, B) \in \V(f_1, f_2, F)$ with $x_2\neq x_1$. This means we have $P_1 \neq \pm P_2$ and if $F(P_1, P_2)=0$, by \Cref{Akbary TFAE} exactly one of $P_1$ and $P_2$ is singular $\bmod \theta$. Interchanging $P_1$ and $P_2$ if necessary, we may assume $P_1$ is singular $\bmod \theta$. This gives us $\partial f_1/\partial x_1 = \partial f_1/\partial y_1 = 0$, therefore showing 
\[
\V(f_1, f_2, F) \subset \V\left(\frac{\partial f_1}{\partial x_1}, \frac{\partial f_1}{\partial y_1}, f_1, f_2\right)
\]
over $\overline{\F_{\theta}}$. 

Now consider the $6$--dimensional affine space with coordinates $\{x_1, x_2, y_1, y_2, A, B\}$ over $\bar{\F}_{\theta}$. We know that $\V(f_1, f_2)$ is $4$--dimensional, so $\V(f_1, f_2, F)$ has dimension $4$ or $3$. However, also note that 
\begin{align*}
&\quad \V\left(\frac{\partial f_1}{\partial x_1}, \frac{\partial f_1}{\partial y_1}, f_1, f_2\right) \\
&=
\{(x_1, x_2, y_1, y_2, A, B) \in \mathbb{A}^6 \colon f_1(P_1, P_2)=f_2(P_1, P_2) =\frac{\partial f_1}{\partial x_1}(P_1, P_2) = \frac{\partial f_1}{\partial y_1}(P_1, P_2)= 0\} \\
&\subseteq 
\{(x_1, x_2, y_1, y_2, A, B) \in \mathbb{A}^6 \colon  -16(4A^3+27B^2)=0, y_2^2=x_2^3+Ax_2+B, -3x_1^2-A= 0, 2y_1= 0 \}.
\end{align*}
SInce $\overline{\F_{\theta}}$ is not a characteristic $2$ field, this shows that $\V(\frac{\partial f_1}{\partial x_1}, \frac{\partial f_1}{\partial y_1}, f_1, f_2)$ has dimension less than or equal to $2$ necessarily, so it is impossible to contain $\V(f_1, f_2, F)$ as a subset. By contradiction, such $F(x_1, x_2, y_1, y_2, A, B)$ cannot exist. 
\end{proof}

This lemma tells us that after reduction modulo any odd prime $\p$, $R_1(X, x_2)$ and $R_2(X, x_2)$ have no common factor and therefore the resultant $\Res_{x_2}(R_1, R_2)(X)$ is not the zero polynomial modulo $\p$. To give a full description of all the irreducible factors of $\Res_{x_2}(R_1, R_2)(X)$, we will also need to use \Cref{singular vanish} and \Cref{zero condition}.
 
\begin{prop}\label{Shape}
Let $E$ be \textcolor{red}{a cubic} curve defined by the short Weierstrass equation (\ref{short EC}), so it has discriminant $\Delta=-16(4A^3+27B^2)$. Let $R_1, R_2$ be as in \Cref{R1 R2}. Then $\Res_{x_2}(R_1, R_2)(X)$ can be expressed as
\begin{equation*}
\Res_{x_2}(R_1, R_2)(X) = \pm X^{L_2(\bm{v})}\Delta^{L_1(\bm{v})}g_{t}(X),
\end{equation*}
where $L_i\colon \Z^2 \rightarrow \Z$ is a function in $v_1, v_2$ for $i=1, 2$, and $g_t(X) \in \Z[A, B, t][X]$ is a polynomial that \textcolor{red}{vanishes exactly on the closure of the set of $(A, B, X)$ such that $E$ is smooth and there exists} an appropriate pair of points $\bm{P}=(P_1, P_2) \in E(\overline{\Q})^2$ satisfying $x([v_1]P_1 \pm [v_2]P_2)=t$ and $[v_1]P_1 \mp [v_2]P_2=O$.
\end{prop}

\begin{proof}
From previous sections, the resultant $\Res_{x_2}(R_1, R_2)(X)$ vanishes if and only if at least one of the following conditions holds:
\begin{enumerate}
\item \label{cond 1}There exists a pair of points $\bm{P}=(P_1, P_2) \in E(\overline{\Q})^2$ such that $X=x_2-x_1=0$, corresponding to $P_1=\pm P_2$. This follows from the definitions of $\Psi_{\bm{v}}^{*}$ and $\Phi_{\bm{v}}^{*}$ in (\ref{new poly}).
\item \label{cond 2}The curve $E\colon y^2=x^3+Ax+B$ is singular, or equivalently $\Delta=-16(4A^3+27B^2)=0$, as shown in \Cref{singular vanish}.
\item \label{cond 3}There exists a pair of appropriate points $\bm{P}=(P_1, P_2) \in E(\overline{\Q})^2$, we have $x([v_1]P_1 \pm [v_2]P_2)=t$ and $[v_1]P_1 \mp [v_2]P_2=O$. This follows from \Cref{zero condition}.
\end{enumerate}

\textcolor{red}{Now consider the ambient polynomial ring $\overline{\Q}[X, A, B]$ and the ambient space $\mathbb{A}_{\overline{\Q}}^3=\{(X, A, B)\colon X, A, B \in \overline{\Q}\}$. Conditions (\ref{cond 1}) and (\ref{cond 2}) define the Zariski closed sets $Z_1=\V(X)$ and $Z_2=\V(\Delta)$ respectively. Let $S$ denote the locus described by condition (\ref{cond 3}) on $X\Delta\neq0$. After extracting all powers of $X$ and $\Delta$ from $\Res_{x_2}(R_1, R_2)(X)$, denote the remaining factor by $g_t(X)$. By construction, $\V(g_t)$ is precisely the Zariski closure of $S$, since $S=\V(g_t)\setminus(Z_1\cup Z_2)$ and no irreducible component of $\V(g_t)$ can be contained in $Z_1\cup Z_2$.} Finally, \Cref{algebraic geometry} guarantees that there are no more other possible irreducible factors. Therefore, we have shown that 
\[
\Res_{x_2}(R_1, R_2)(X) = CX^{L_2(\bm{v})}\Delta^{L_1(\bm{v})}g_{t}(X), C\in \Z.
\]
Using the argument from \cite[Lemma 1.7.11(b)]{Connell} (or \cite{hindry_cassels} ), we claim that $C=\pm 1$. Suppose $C$ has an odd prime factor $p$, then $\Res_{x_2}(R_1, R_2)(X)$ vanishes identically $0$ in characteristic $p$. By \Cref{Akbary TFAE}, this forces $4A^3+27B^2$ to be a zero polynomial in $\F_p[A, B]$, which is clearly false. 

If $p=2$, the short Weierstrass equation (\ref{short EC}) is not an appropriate model, since every cubic curves of such form is singular in characteristic $2$. Instead, we apply the same argument to an integral general Weierstrass equation 
\[
E\colon  y^2+a_1xy+a_3y=x^3+a_2x^2+a_4x+a_6, a_i \in \OO_K.
\]
The discriminant of the general Weierstrass equation cannot be a zero polynomial in characteristic $2$, otherwise all elliptic curves are singular in in characteristic $2$ (for example, $y^2+xy=x^3+1$ defines a nonsingular cubic over $\F_2$). Hence, the same argument rules out $2\mid C$ and we conclude that $C=\pm1$.
\end{proof}

\begin{remark}
Although the proposition is stated for the short Weierstrass model, the definitions of $\Psi_{\bm{v}}^{*}$ and $\Phi_{\bm{v}}^{*}$, as well as the construction of the resultant in the previous sections, do not depend on the choice of the short Weierstrass model. Hence, there is no compatibility issue in treating the case $p=2$ using an integral general Weierstrass equation.
\end{remark}

\subsection{Valuations of net polynomials and the discriminant}
By the definition of the resultant and \Cref{Shape}, there exists $a(X, x_2), b(X, x_2) \in \Z[A, B][X, x_2]$ such that 
\[
R_1(X, x_2)a(X, x_2)+R_2(X, x_2)b(X, x_2)=\pm X^{L_2(\bm{v})}\Delta^{L_1(\bm{v})}g_t(X).
\]
Let $\p$ be a prime in $K$ and $\nu$ be the associated valuation. If we evaluate them at a specific pair $(X^{\circ}, x_2^{\circ}) \in K^2$, then we know $\nu\left(R_1(X^{\circ}, x_2^{\circ})\right), \nu\left(R_2(X^{\circ}, x_2^{\circ})\right)>0$ implies that the $\p$--valuation of $X^{L_2(\bm{v})}\Delta^{L_1(\bm{v})}g_t(X)$ has to be positive as well. 

We have assumed that $P_1\not\equiv \pm P_2 \bmod \p$, so $X^{\circ}\not\equiv 0 \bmod \p$. As a result, we expect that the discriminant $\Delta$ is the only factor among the irreducible factors of the resultant that detects when $\Psi_{\bm{v}}^2(\bm{P})$ and $\Phi_{\bm{v}}(\bm{P})$ both have positive $\p$--valuation at primes of bad reduction. Therefore, the goal of this subsection is to show that $g(X^{\circ})$ and $\Delta$ cannot both have positive $\p$--valuation simultaneously at primes of bad reduction, or equivalently primes that divide both $\Psi_{\bm{v}}^2(\bm{P})$ and $\Phi_{\bm{v}}(\bm{P})$ . 

As stated in the previous section, our method for showing this is to choose $t \in \overline{\Z}$ carefully such that the polynomial $g_{t}(X)$ does not have a root over $\overline{F_{\p}}$. For this reason, we will first need a version of \Cref{Shape} for algebraic closures of finite fields. 

\begin{lemma}\label{monic}
The polynomial $g_{t}(X) \in \Z[A, B, t][X]$ is a monic polynomial in the variable $X$.
\end{lemma}

\begin{proof}
It suffices to prove that if $t, A, B$ are all algebraic integers and $X$ is a root of $g_t(X)$ in $\C$, then $X$ is an algebraic integer. We will prove by contradiction, so suppose $X$ is not an algebraic integer, then there exists a prime $\q \in \OO_{K(\bm{P})}$ such that $X \equiv \infty \bmod \q$, where $K(\bm{P})$ is a finite extension of $K$ formed by adjoining the coordinates of $P_1, P_2 \in E(\overline{\Q})$. Without loss of generality, this implies that $P_1 \equiv \infty \bmod \q$. From (\ref{zero condition}), if $X$ is a root of $g_t(X)$, then we know there are $P_1, P_2 \in E(\overline{\Q})$ that satisfy 
\[
x\left(\pm[v_1]P_1\pm[v_2]P_2\right)=t, \text{ and } \pm[v_1]P_1\mp[v_2]P_2=O.
\]
Together with $P_1 \equiv \infty \bmod \q$, this implies that $t\equiv x\left([v_1]P_1+[v_2]P_2\right)\equiv \infty \bmod \q$, so $t$ cannot be an algebraic integer. 
\end{proof}

\begin{prop}[A local version of \Cref{Shape}]\label{local shape}
Let $\q$ be a prime lying above $\p$. For all $\overline{A}, \overline{B}, \in \F_{\p}$ and $\overline{X}$ , $\overline{t} \in \overline{\F_{\p}}$, we have $g_{\overline{t}}(\overline{X})=0$ if and only if there exists a pair of points $P_1, P_2 \in E(\overline{Q})$ such that $(\overline{A}, \overline{B}, \overline{P_1}, \overline{P_2})$ is in the Zariski closure of the union of the sets below.
\begin{equation}\label{local congruence}
\begin{aligned}
C_1&=\{\Delta(E)\not\equiv 0 \bmod \q\}, \\
C_2&=\{[v_1]P_1\pm[v_2]P_2\equiv O \bmod \q\}, \text{ and }\\
C_3&=\{x\left([v_1]P_1\mp[v_2]P_2\right)\equiv \overline{t} \bmod \q\}.
\end{aligned}
\end{equation}
\end{prop}

\begin{remark}
\Cref{Shape} guarantees $P_1, P_2 \in E(\overline{\Q})$ are a pair of appropriate, hence non-singular point. However, their reductions modulo $\q$ need not be. Therefore, the expressions in (\ref{local congruence}) are interpreted by first performing the group operations in characteristic zero and then reducing the resulting points or coordinates modulo $\q$.
\end{remark}

\begin{proof}
Let $\overline{A}, \overline{B}, \overline{X} \in \F_{\p}$ and $\overline{t} \in \overline{\F_{\p}}$ be such that $g_{\overline{t}}(\overline{X}) = 0$. We wish to show that there exist $A, B \in \OO_K$ and $X, t \in \overline{\Z}$ such that $A \equiv \overline{A}$, $B \equiv \overline{B}$, and $X \equiv \overline{X} \bmod \q$ for some prime ideal $\q$ lying above $\p$.

To construct $X$ and $\q$, choose $Z \in \OO_K$ such that $Z \equiv \overline{X} \bmod \p$ and $g_{t}(Z) \equiv 0 \bmod \p$. Consider the polynomial $g_t(Y + Z) \in \Z[A, B, t][Y]$. By \Cref{monic}, $g_t(Y + Z)$ is monic, so it factors over $\overline{\Q}$ as
\[
g_{t}(Y+Z) = \prod_{i=1}^{n} \left(Y - (Y_i - Z)\right),
\]
where $Y_1, \ldots, Y_n \in \overline{\Z}$ are the roots of $g_{t}(Y)$. Let $L = K(Y_1, \ldots, Y_n)$ be the splitting field over $K$, and let $\q \subset \OO_L$ be a prime ideal lying above $\p$. Evaluating $g_t(Y + Z)$ at $Y = 0$ yields
\[
g_t(Z) = (-1)^n \prod_{i=1}^{n} (Y_i - Z).
\]
Since $g_t(Z) \equiv 0 \bmod \p$, we have $g_t(Z) \in \q$, which implies $\prod_{i=1}^{n} (Y_i - Z) \in \q$ and there exists some $Y_0 \in \{Y_1, \ldots, Y_n\}$ such that $Y_0 - Z \equiv 0 \bmod \q$. Setting $X = Y_0+Z \in \overline{\Z}$, we obtain 
\[
g_{t}(X)=g_t(Y_0+Z)=g(Y_0)=0 \text{ with } X\equiv Y_0+Z \equiv Z \bmod \q.
\]
Having established suitable values $X, t \in \overline{\Z}$ and choose $A, B \in \OO_K$ such that $E/K$ is defined by a minimal integral equation, we apply \Cref{Shape} to obtain a pair of points $(P_1, P_2) \in E(\overline{\Q})^2$ satisfying 
\[
x([v_1]P_1 \pm [v_2]P_2) = t \text{ and } [v_1]P_1 \mp [v_2]P_2 = O.
\]
Reducing these relations modulo $\q$, we have found a set of $(A, B, X)$ that is contained in the Zariski closure of (\ref{local congruence}).
\end{proof}

With \Cref{local shape}, we can now state the criteria for choosing such $t \in \overline{\Z}$. 

\begin{prop}\label{separate}
Let $k\in \N$ be a natural number, $\bm{v}=(v_1, v_2)\in \Z^2$ and $\bm{P}=(P_1, P_2) \in E(\overline{\Q})^2$. Let $X=x(P_2)-x(P_1)=x_2-x_1$. Suppose the reduction of $P_1$ and $P_2$ satisfies $\overline{P_1}\neq \pm \overline{P_2}$ and $\Psi_{\bm{v}}^{*}(\bm{P})\equiv \Phi_{\bm{v}}^{*}(\bm{P}) \equiv 0 \bmod \p^k$ for a prime ideal $\p \subset \OO_K$, then there exists an algebraic integer $t \in \overline{\Z}$ such that no prime over $\p$ divides $X^{L_2(\bm{v})}g_t(X)$. 
\end{prop}

\begin{proof}
Let $\nu$ be the valuation associated to the prime $\p$ and assume $\Psi_{\bm{v}}^{*}(\bm{P})\equiv \Psi_{\bm{v}}^{*}(\bm{P}) \equiv 0 \bmod \p^k$. Replacing $y_i$--variables using the Weierstrass equation $y_i^2=x_i^3+Ax_i+B$ will not reduce the $\p$--valuation. Hence, by the definition of $R_1(X, x_2)$ and $R_2(X, x_2)$ in \Cref{R1 R2}, for every $t \in \Z$ we have $w(R_1), w(R_2)\geq k$ for some valuation $w$ extending $\nu$. This also implies that $w\left(X^{L_2(\bm{v})}\Delta^{L_1(\bm{v})}g_t(X)\right) \geq k$, which follows from \Cref{Shape}. Since we have also assumed the reduction of $P_1$ and $P_2$ satisfies $\overline{P_1}\neq \pm \overline{P_2}$, we have $w(X)=0$. Hence, it remains to show the existence of $t\in \overline{\Z}$ such that $w(g_t(X))=0$. 

We first assume that both $P_1$ and $P_2$ have good reduction modulo $\p$, therefore $\overline{P_1}, \overline{P_2} \in E(\F_{\p})$. 

Suppose there exists $t\in  \overline{\Z}$ such that $w(g_t(X))>0$, then by \Cref{local shape}, there exist $\overline{P_1}, \overline{P_2} \in E(\overline{\F_{\p}})$ such that 
\begin{equation}\label{local congruence 2}
\begin{aligned}
\overline{x_2}-\overline{x_1}&=\overline{X}, \\
[v_1]\overline{P_1}\pm[v_2]\overline{P_2}&=\overline{O}, \\
x\left(\pm[v_1]\overline{P_1}\mp[v_2]\overline{P_2}\right)&=\overline{t},
\end{aligned}
\end{equation}
where $\overline{O}$ denotes the identity element on the reduced curve. 

Let $\q$ be a prime lying above $\p$; we split the proof into two cases.

\textbf{Case I:} $|v_1|\neq |v_2|$ \\
In this case, (\ref{local congruence 2}) tells us that 
\[
x_2\equiv X+x_1 \bmod \q \text{ and } x\left([2v_1]P_1\right)\equiv x\left([2v_2]P_2\right)\equiv t \bmod \q. 
\]
Using \Cref{net poly def} (with $n=1$, or equivalently the definition of division polynomials), the above congruence conditions can also be rewritten as 
\[
x\left([2v_1]P_1\right) \equiv \frac{\phi_{2v_1}(x_1)}{\psi_{2v_1}^2(x_1)}\equiv x\left([2v_2]P_2\right) \equiv \frac{\phi_{2v_2}(X+x_1)}{\psi_{2v_2}^2(X+x_1)} \equiv t \bmod \q.
\]
From \cite[Theorem A]{AyadS}, we know that $\psi_{2v_1}^2(x_1)\equiv \psi_{2v_2}^2(X+x_1)\equiv 0 \bmod \q$ if and only if one of the following holds:
	\begin{enumerate}
	\item $[2v_1]P_1 \equiv [2v_2]P_2 \equiv O \bmod \q$; or 
	\item $P_1, P_2$ are both singular modulo $\q$; or
	\item Without loss of generality,  $[2v_1]P_1\equiv O \bmod \q$ and $P_2$ is singular modulo $\q$. 
	\end{enumerate}
The first possibility cannot be true, as it implies $t \equiv \infty \bmod \q$, violating the condition that $t \in \overline{\Z}$. If the second situation were to occur, then $X\equiv 0 \bmod \q$ since there is only one singular point modulo $\q$, contrary to our assumption. Finally, the last possibility implies that $t \in \overline{\Z}$ has to satisfy $t \equiv \infty \bmod \q$ and $\overline{t}$ is the $x$--coordinate of the singular point modulo $\q$ simultaneously, which cannot happen. Therefore, we may rearrange the above as
	\begin{equation}\label{local div poly}
	\phi_{2v_1}(x_1)\psi_{2v_2}^2(X+x_1)-\phi_{2v_2}(X+x_1)\psi_{2v_1}^2(x_1) \equiv 0 \bmod \q.
	\end{equation}
Remain to show that this polynomial is nonzero. To do so, we recall from \cite[Chapter III, Exercise 3.7(b)]{Silverman1} that as polynomials in one variable, 
	\begin{align*}
	\phi_n(x)&=x^{n^2}+(\text{lower order terms}),\\
	\psi_n^2(x)&=n^2x^{n^2-1}+(\text{lower order terms}).
	\end{align*}

If $|v_1|\neq |v_2|$ and $X\equiv x(P_2)-x(P_1) \not\equiv 0 \bmod \q$, then the polynomials in (\ref{local div poly}) cannot cancel out each other due to the difference in degree. Therefore, (\ref{local div poly}) is a polynomial in the variable $x_1$ over $\overline{\F_p}$ of degree at most $4v_1^2+4v_2^2-1$. Since $X$ is a fixed value, every solution of (\ref{local congruence 2}) immediately determines a value of $x_2 \bmod \q$, hence $t \bmod \q$ as well and therefore there are at most $4v_1^2+4v_2^2-1$ distinct values of $x_1$ as roots of the polynomial (\ref{local div poly}). This gives us a finite set of possible values for $\overline{t}\in \overline{\F_{\p}}$ such that (\ref{local congruence 2}) is soluble for some $\overline{P_1}, \overline{P_2} \in E(\overline{\F_{\p}})$, so we can simply choose $t\in \overline{\Z}$ that its reduction $\overline{t}$ is outside this finite set. 

\textbf{Case II:} $|v_1|=|v_2|=v$ \\
In this case, (\ref{local congruence 2}) tells us that there exist $\overline{P_1}, \overline{P_2} \in E(\overline{\F_{\p}})$ such that 
\begin{align}
\overline{X}=\overline{x_2}-\overline{x_1} &\neq 0, \label{eq:X-nonzero}\\
[v](\overline{P_1}\pm \overline{P_2}) &= \overline{O}, \label{eq:plus-minus}\\
x\left([v](\overline{P_1}\mp \overline{P_2})\right) &= \overline{t}. \label{eq:t-condition}
\end{align}
Without loss of generality, we assume $[v](\overline{P_1}+ \overline{P_2}) \equiv O$ and $[v](\overline{P_1}- \overline{P_2})=\overline{t}$ with $\overline{X}$ being fixed here. Set $\overline{Q}=\overline{P_1}+\overline{P_2}$, then $\overline{Q} \in E[v]$. From (\ref{eq:X-nonzero}), $\overline{Q}\neq \overline{O}$, so there are at most $v^2-1$ choices for $\overline{Q}$. 

Next, for each choice of $\overline{Q} \in E[v]$, we consider the rational function
\[
f_{\overline{Q}}(\overline{P_1})=x(\overline{Q}-\overline{P_1})-x(\overline{P_1})-\overline{X} \in \overline{\F_{\p}}(E)^*.
\]
This is indeed a rational function: we have shown that $\overline{Q}\neq \overline{O}$ and assumed $\overline{X}\neq0$, additionally the $x$--coordinate function and any nonzero constant function are rational functions.

As a rational function, $f_{\overline{Q}}(\overline{P_1})$ has $2$ poles at $\overline{P_1}=\overline{Q}$ and $\overline{P_1}=\overline{O}$ each respectively. Since we necessarily have $\deg(\operatorname{div}(f_{\overline{Q}}))=0$, this implies that $f_{\overline{Q}}(\overline{P_1})$ has $4$ zeros. Hence, for each choice of $Q$, there are at most $4$ distinct choices for $\overline{P_1}$ satisfying (\ref{eq:X-nonzero}) and overall a maximum of $4(v^2-1)$ choices for $\overline{P_1}$ satisfying all the congruence conditions simultaneously. Since $\overline{X}$ is fixed, every choice of $\overline{P_1}$ immediately determines a value of $\overline{P_2}$ using (\ref{eq:X-nonzero}). As in Case I, this gives us a finite set of possible values for $\overline{t}\in \overline{\F_{\p}}$ such that (\ref{eq:t-condition}) is soluble for some $\overline{P_1}, \overline{P_2} \in E(\overline{\F_{\p}})$, so we can simply choose $t\in \overline{\Z}$ that its reduction $\overline{t}$ is outside this finite set. 

Next, we handle the case where exactly one of $P_1$ and $P_2$ is singular mod $\q$, say $P_2$. Denote the reduction of $P_1$ and $P_2$ modulo $\q$ by $\overline{P_1}$ and $\overline{P_2}$ respectively. 

We have assumed that $\overline{X}\neq 0$, so $P_1$ cannot have singular reduction and therefore the group law of $E(\overline{\F_{\p}})$ is applicable to $\overline{P_1}$. By \Cref{local shape}, $\overline{P_1}$ and $\overline{P_2}$ satisfy 

\begin{equation*}
\begin{aligned}
\overline{x_2}-\overline{x_1}&=\overline{X}, \\
\overline{[v_1]P_1\pm[v_2]P_2}&=\overline{O}, \\
\overline{x\left(\pm[v_1]P_1\mp[v_2]P_2\right)}&=\overline{t}.
\end{aligned}
\end{equation*}

Since $[v_1]\overline{P_1}$ is a non-singular point on $E(\overline{\F_{\p}})$, so as $\overline{[v_2]P_2}$ by the second condition, therefore the above relations imply that we have $x([2v_1]P_1)\equiv t \bmod \q$. Denote the singular point on $E(\overline{\F_{\p}})$ by $P^*=(x^*, 0)$, then the assumption implies that $x(P_2)\equiv x^* \bmod \q$ and therefore requiring $x(P_1)\equiv x^* -X \bmod \q$. There are at most $2$ points on $E(\overline{\F_{\p}})$ with $x(P_1)\equiv x^* -X \bmod \q$, say $\overline{P_1}=\pm Q$, hence giving at most $2$ possible values for $\overline{t}\in \overline{\F_{\p}}$ such that the above congruences is soluble. We may choose $t\in \overline{\Z}$ that its reduction $\overline{t}$ is outside this finite set, then we cannot have $w(g_t(X))>0$. 

Overall, we showed that for whenever $\overline{X}\neq0$ in $\overline{F_{\p}}$, there exists $t\in \overline{\Z}$ such that $w(X^{L_2(\bm{v})}g_t(X))=0$. For this choice of $t \in \overline{\Z}$, we have the implication 
\[
w\left(X^{L_2(\bm{v})}\Delta^{L_1(\bm{v})}g_t(X)\right) \geq k \implies w\left(\Delta^{L_1(\bm{v})}\right) \geq k,
\]
therefore completing the proof. 
\end{proof}

\begin{remark}\label{X=0 impossible}
If $\overline{X}=0$ in $\overline{F_{\p}}$, it is possible that no such $t \in \overline{Z}$ exists. Specifically, if $|v_1|=|v_2|$, then the set of congruences in (\ref{local congruence 2}) collapses to a system of $2$ congruences only: 
\begin{align*}
\overline{x_2}-\overline{x_1}=\overline{X}&=0, \\
x\left([2v]\overline{P_1}\right)&=\overline{t}.  
\end{align*}
Furthermore, the multiplication map $[v]\colon E \rightarrow E$ is surjective, so any value $\overline{t} \in \overline{\F_{\p}}$ is attained and our method fails. This is one of the reasons why we have the condition $P_1\not\equiv \pm P_2 \bmod \p$ in our main theorem. 
\end{remark}

\begin{proof}[Proof of \Cref{main result}]
Let $\p$ be a prime and $P_1, P_2 \in E(\overline{\Q})$ such that $P_1, P_2 \not\equiv O \bmod \p$ and $P_1\not\equiv \pm P_2 \bmod \p$. Under these assumptions, \cite[Proposition 1.6]{Akbary} asserts that $\nu\left(\Phi_{(v_1, v_2)}(\bm{P})\right), \nu\left(\Phi_{(v_1, v_2)}(\bm{P})\right)\geq 0$. If $\p$ divides both $\Psi_{\bm{v}}^{*}$ and $\Phi_{\bm{v}}^{*}$, then \Cref{Akbary TFAE} tells us that $\p$ divides the discriminant $\Delta$ necessarily.

Under the assumption of $P_1\pm P_2\not\equiv O \bmod \p$, \Cref{separate} shows that if $P_1 \neq \pm P_2 \bmod \p$, then there exists $t\in \overline{\Z}$ such that $g_{t}(X)$ and $\Delta$ cannot share positive $\p$--valuation simultaneously. Hence, the valuation of $\Delta$ in $\Res_{x_2}(R_1, R_2)(X)$ bounds the common valuation of $\Psi_{\bm{v}}^{2}$ and $\Phi_{\bm{v}}$ associated to odd primes of bad reduction from above.

Note that $R_1(X, x_2)$ and $R_2(X, x_2)$ are homogenous in weight, where $X, x_2, A$ and $B$ have weights $2, 2, 4$ and $6$ respectively. Recall from \Cref{Section 4} that $\Psi_{\bm{v}}$ and $\Phi_{\bm{v}}$ are homogenous of weight $g(\bm{v})$ and $2(v_1^2+v_2^2-v_1v_2)$ respectively. Thus, we may deduce that
\begin{align*}
\Deg_{x_2}(R_1)
&\leq \frac{1}{2}(8)\left((v_1^2+v_2^2-v_1v_2-1)+2(v_1^2+v_2^2-v_1v_2)\right) \\
&=4\left(3(v_1^2+v_2^2-v_1v_2)-1\right); \text{ and} \\
\Weight(R_2)
&=4\cdot2(v_1^2+v_2^2-v_1v_2)=8(v_1^2+v_2^2-v_1v_2).
\end{align*}

Since $R_1$ and $R_2$ are both homogenous in weight, so is their resultant by \Cref{resultant properties}. Hence, we can bound the multiplicity of $\Delta$ in $\Res_{x_2}(R_1, R_2)(X)$ by assuming that $\Delta$ is the only term contributing to the weight. This bounds the multiplicity of $\Delta$ in $\Res_{x_2}(R_1, R_2)(X)$ from above by
\[
\frac{32}{12}\left(3(v_1^2+v_2^2-v_1v_2)-1\right)(v_1^2+v_2^2-v_1v_2).
\]
Using the ineuqality $2|v_1v_2|\leq v_1^2+v_2^2$, we obtain
\begin{align*}
\text{Multiplicity of $\Delta$ in $\Res_{x_2}(R_1, R_2)(X)$}
&\leq \frac{8}{3}\left(\frac{9}{2}v_1^2+\frac{9}{2}v_2^2-1\right)\left(\frac{3}{2}v_1^2+\frac{3}{2}v_2^2\right)\\
&\leq 18(v_1^2+v_2^2)^2-4(v_1^2+v_2^2).
\end{align*}
This proves the theorem.
\end{proof}

\begin{remark}
We have computed several examples and believe that an optimal upper bound for the multiplicity of $\Delta$ in $\Res_{x_2}(R_1, R_2)(X)$ is 
\[
\frac{8}{3}\left(v_1^2(v_1^2-1)+v_2^2(v_2^2-1)\right).
\]
However, we have not succeeded in proving this bound and are therefore content with the current upper bound. 
\end{remark}